\documentclass[11pt,a4paper]{article}

\usepackage{graphicx} 
\usepackage[pdftex]{pict2e}
\usepackage[dvipsnames]{xcolor}
\usepackage{stix}
\usepackage{appendix}
\input{epsf.sty}
\usepackage{epsfig}
\headheight\baselineskip\headsep2\baselineskip\textwidth210mm
\advance\textwidth -2in\textheight287mm\advance\textheight-2in
\label{seq.sub.gdn}
\usepackage[utf8]{inputenc}
\usepackage{tikz}
\usetikzlibrary{shapes.geometric, positioning, shapes, arrows}
\usepackage{amsfonts}
\usepackage{amsthm}
\usepackage{float}
\usepackage{graphicx}
\usepackage{amsmath}
\usepackage{lscape}
\usepackage{color}
\usepackage{multirow}
\usepackage[pdftex]{pict2e}
\usepackage[dvipsnames]{xcolor}
\usepackage{stix}

\definecolor{c30}{rgb}{0,0,1}

\newtheorem{theorem}{Theorem}[section]
\newtheorem{lemma}[theorem]{Lemma}
\newtheorem{proposition}[theorem]{Proposition}

\theoremstyle{definition}

\newtheorem{example}[theorem]{Example}
\newtheorem{remark}[theorem]{Remark}
\numberwithin{equation}{section}
\title{On a finite quasi birth-death process with catastrophes\\
	and its diffusion approximation}
\author{Giulia Di Nunno \thanks{Department of Mathematics, University of Oslo;
		Department of Business and Management Science, NHH Norwegian School of Economics, giulian@math.uio.no} \and Barbara Martinucci \thanks{Department of Mathematics, University of Salerno, bmartinucci@unisa.it} \and Serena Spina \thanks{Department of Mathematics, University of Salerno, sspina@unisa.it}}
\date{\today}

\begin{document}
	
	\maketitle

\begin{abstract}
We study a multi-type Ehrenfest process modeled as a finite quasi-birth-death (QBD) process. We assume that the transitions are allowed only to the two adjacent levels of the same phase and are 
characterized by linear rates. The crucial element lies in the phase switching mechanism at the origin, which is governed by an irreducible stochastic matrix.
The process evolution is interrupted by catastrophic events, whose occurrences are controlled by a Poisson process. Each catastrophe resets the system state to zero, initiating a new cycle of evolution 
until the next resetting event. We conduct a comprehensive analysis, addressing both the transient and long-term behavior of this process. Furthermore, we derive a diffusive approximation, 
by proving its convergence to a reflected Ornstein-Uhlenbeck jump diffusion process.
\end{abstract}


\section{Introduction}
\label{sec1}

Stochastic models of population dynamics, particle systems, and chemical kinetics frequently rely on structured Markov chains, such as birth-death (BD) processes, due to their tractability and rich probabi\-listic properties. 
While a traditional birth–death process allows transitions only between neighboring states, a quasi birth-death process (QBD) extends this idea to a two-dimensional setting, where states are organized into levels and phases.
Each level represents a quantity of interest - such as the number of customers in a queue or jobs in a system - while each phase captures additional system conditions, such as the status of servers or the environment. 
Transitions can occur within a level (phase changes), to the next higher level (a birth), or to the next lower level (a death). The properties of finite QBD processes have been studied extensively in continuous and discrete time. 
For a comprehensive discussion, see, the monograph \cite{Latouche}. In \cite{Latouche2002} a key aspect of quasi-birth-and-death (QBD) processes is analyzed, that is the probabilities of reaching level zero from level one, which are captured in the $q$-matrix, called $G$. If the process is recurrent, $G$ is a stochastic matrix. If it is not, $G$ becomes sub-stochastic, and a second, fully stochastic solution, exists.
\par
The transition rate matrix of a QBD process has a block-tridiagonal structure, which allows efficient analytical and numerical methods - such as the matrix-geometric method - to compute steady-state probabilities and performance measures. A common methodology to study such kind of processes is based on the use of first-step analysis, which allows us to derive systems of linear equations that characterize the distribution of the stochastic descriptors. Then one could solve the resulting systems of linear equations with specialized numerical methods or develop efficient algorithmic procedures that exploit the block-tridiagonal structure of the original QBD process. 
In \cite{DiCrescenzo2025}, for instance, the authors exploit a special property of the block-tridiagonal matrix to develop some computational algorithms. Such results are employed in two epidemic models: the SIS model for 
horizontally and vertically transmit\-ted diseases and the SIR model with constant population size. 
\par
An overview on the most recent results in the field of QBD processes with respect to numerical methods for the steady-state analysis of such processes is presented in \cite{Ost}.
The author reviewes the most efficient algorithms for the state-state analysis, comparing different algorithms both from a theoretical point of view and
in a practical case study in order to highlight the strengths and weaknesses of the various approaches. 
\par
Incorporating catastrophic events in BD end QBD processes has become a crucial aspect of stochastic population modeling, particularly in ecology, but also across various other fields such as economics, chemistry, and telecommunications.
The inclusion of large, sudden jumps  help explaining certain features observed in financial markets, such as returns on stock indices and currency exchange rates \cite{Takahashi}. In \cite{Dharmaraja},  the authors consider a continuous-time Ehrenfest model defined over the integers from $-N$ to $N$, and subject to catastrophes occurring at constant rate. The effect of each catastrophe instantaneously resets the process to state $0$. 
Another study of BD process with catastrophes is given in \cite{Cairns}, where the transition rates vary with population size in a general way. Despite this broader framework, the authors still derive explicit results for key quantities as the expected time to extinction.
\par
The study in \cite{Baumann} is focused on continuous-time level-dependent (LDQBD) processes that include catastro\-phic events, allowing the population level to suddenly drop to zero. These processes are useful for modeling queueing systems and other population dynamics. The authors develop a matrix analytic algorithm to compute the stationary distribution of such models, extending existing methods for LDQBD processes without catastrophes. The algorithm is applied to analyze $M/M/c$ queues in random envi\-ronments with state-dependent rates and catastrophes. 
\par
The manuscript \cite{Dudin} introduces a method to derive the generator of a multidimensional continuous-time Markov chain that models a queueing system with the possibility of disasters, based on the known generator of a standard QBD process without disasters. For the level-independent case, the stationary distribution of the system without disasters is obtained in matrix-geometric form.
 The interplay between gradual evolution and sudden reset events makes these models especially relevant for real-world phenome\-na characterized by persistence interspersed with abrupt changes.
\par
In the present paper we consider  an extension of \cite{DiCrescenzo2022}, by considering a finite quasi birth-death process subject to catastrophes. We assume that catastrophic events occur 
according to a Poisson process and that the resetting instantaneously takes the system to the state $0$. This mechanism gives rise to a QBD process with catastrophes as a model particularly suited to capture abrupt systemic changes such as collapses in population models, system failures in queueing theory, or information resets in communication networks.
Such finite QBD process can be  considered as a multi-type extension of the continuous-time Ehrenfest model, with state space formed by the integers $0,1,2,\ldots,N$ of 
$d$ lines joined at the origin, in the presence of reflecting boundaries and under the assumption of state-dependent transition rates. The particular case $d = 2$ corresponds to the one-dimensional Ehrenfest model, as described in \cite{Dharmaraja}. 
\par
We develop the mathematical framework for this class of models, analyze their transient and stationary behavior, and provide closed-form expressions for relevant quantities. We employ generating function techniques and matrix-analytic methods to characterize the system evolution and highlight the impact of catastrophic transitions on long-term dynamics. Importantly, while most of the existing literature on QBD processes focuses primarily on asymptotic or steady-state analyses, our results provide novel insights into the finite-time behavior and transient characteristics of such systems, thereby enhancing their applicability to realistic modeling scenarios. 
 Furthermore, we characterize the model long-term evolution by exploring its large-scale behavior through a diffusive approximation, which offers a powerful and tractable analytical framework.
\section{The model}

Let us consider a finite quasi-birth-death (QBD) process ${\bf X}_t\!\!:=\{({\cal N}(t),J(t)), t\geq 0\}$, whose two-dimensional state space is partitioned as $\bigcup_{k=0}^{N} l(k)$, 
where $l(k)=\left\{(k,1),\,(k,2),\ldots,\,(k,d)\right\}$, $k=0,\ldots, {N}$, with $d, {N} \in {\mathbb N}^+$. In each state $(k,j)$, the first coordinate $k$ represents the {\em  level}, 
while the second coordinate $j$ denotes the {\em phase}. Typically, the process ${\cal N}(t)$ counts the number of presences in the system at time $t$, such as 
the population size, the queue length, or some other quantity that evolves through unitary increments or decrements. The component $J(t)$ describes the environmental state at time $t$, or, 
more generally, the internal condition of the system within a given level. 
\par
In the general framework of QBD processes, the one-step transitions from a state $(k,j)$ are restricted to the states belonging either to the same level  $(k,j')$, 
or to one of the two adjacent levels, $(k+1,j')$ and $(k-1,j')$, $j,j'\in D:=\{1,\ldots,d\}$. In the present model, we consider level-dependent one-step transition probabilities and assume that, 
for $k\geq 1$ and $j\in D$, the transitions from the state $(k,j)$ are allowed only to the two adjacent levels within the same phase. However, from the state $(0,j)$ we can have 
one-step transitions to $(1,j'),\, j'\in D$, which are regulated by an irreducible stochastic matrix ${\bf C}=(c_{lj})$, with 
\begin{equation}
	c_{l,j}\geq 0,\qquad \sum_{j=1}^N c_{l,j}=1,\; \forall l \in D.
	\label{c_lj}
\end{equation}

Furthermore, catastrophic events are assumed to occur according to a Poisson process with rate $\xi>0$, i.e. the probability that a catastrophe occurs during a small time 
interval $(t, t+\Delta t)$ is given by $\xi \Delta t + o(\Delta t)$. 

Upon the occurrence of a catastrophe, the level of ${\bf X}_t$ is instantaneously reset to zero, while the phase remains unchanged.
Following such an event, the process resumes its usual evolution until the next catastrophic event takes place.

\begin{figure}[t] 
\begin{center}
\includegraphics[width=8cm]{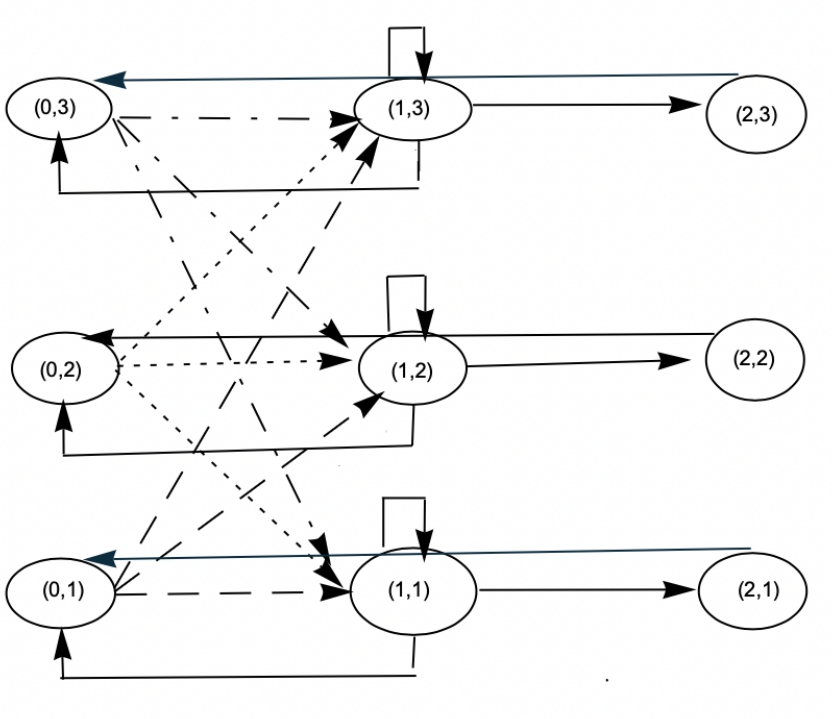}
\caption{State transition diagram in the case $d=3$ and $N=2$.}
\label{Diagram}
\end{center}
\end{figure}

The evolution of ${\bf X}_t$ is regulated by the following rules, where $h>0$ is small: 
\begin{description}
	\item{\em (i) \ } 
	if at time $t$ the process is in state ${\bf X}_t =(0,l)$, $l\in D$, then during the subsequent interval $(t,t+h]$ the process makes a transition to the state 
	$(1,j)$, $j\in D$, with probability $c_{l,j}\lambda N h+o(h)$, or the level remains unchanged with probability $1-\lambda N h+o(h)$, where $\lambda>0$;
	\item{\em (ii) \ } 
	if at time $t$ the process is in state ${\bf X}_t =(k,j)$, $k=1,2,\ldots,N-1$, $j\in D$, then
	during the subsequent interval $(t,t+h]$ one of the following transitions may occur: the process moves to the state $(k-1,j)$ with probability 
	$\mu (N+k) h+o(h)$, where $\mu>0$; it moves to the state $(k+1,j)$ with probability $\lambda (N-k)  h+o(h)$; or a catastrophic event occurs, causing a 
	transition to the state $(0,j)$ with probability $\xi h+o(h)$. Otherwise, the level remains unchanged with probability
	$1-[\mu (N+k)+\lambda (N-k)+\xi] h+o(h)$;
	\item{\em (iii) \ } 
	if at time $t$ the process is in state ${\bf X}_t =(N,j)$, $j\in D$, then 
	during the subsequent interval $(t,t+h]$ one of the following transitions may occur: the process moves to the state $(N-1,j)$ with probability $2N\mu h+o(h)$; 
	a catastrophic event occurs, resulting in a transition to the state $(0,j)$ with probability $\xi h+o(h)$; 
	or the level remains unchanged with probability $1-2N\mu h-\xi h+o(h)$, where $\mu>0$ and $\xi>0$. 
\end{description}
Due to assumptions {\em (i)}, {\em (ii)} and {\em (iii)}, the environment does not influence the birth and death process when the system is non empty, 
whereas the birth and death transition rates depend on the level of the system. 
Moreover, phase transitions are permitted only when the system is empty. Specifically, if at time $t$ the level of ${\bf X}$ is $0$ and the 
the environmental state is $J(t)=l$, $l\in D$, then the probability of a transition to the state $(1,j)$ is given by $\lambda N c_{l,j}$. 
The corresponding state-transition diagram is illustrated in Figure \ref{Diagram} in the particular case $d=3$, $N=2$. In such a picture, the levels are arranged 
as columns, while the rows correspond to the different environmental states.

In order to disclose the infinitesimal generator $Q$ of the system, let us define the transition rates of ${\bf X}_t$ as 
\begin{eqnarray*}
&& \hspace*{-1cm}
q(k,j; k',j')=\lim_{h\rightarrow 0^+} \frac{1}{h} 
{\mathbb P}[({\cal N}(t+h),J(t+h))=(k',j')\,|\,({\cal N}(t),J(t))=(k,j)],
\\
&& \hspace*{2cm}
\quad k, k'\in \{0,\ldots,N\}, j,j'\in D.
\end{eqnarray*}
%
According to the rules {\em (i)}, {\em (ii)} and {\em (iii)}, for $l,\,j\in D$, we have
%
$$
q(0,l;1,j)=c_{l,j} \lambda N,\qquad q(0,l;k,j)=0\quad \hbox{if $k\neq 1$}, 
$$
\begin{equation}
	q(k,j;k+1,j)=\lambda(N-k),\qquad k=1,2,\ldots,N-1, 
		\label{eq:tassi1}
\end{equation}
\begin{equation*}
	q(k,j;k-1,j)=\mu (N+k),\qquad k=1,2,\ldots,N,
\end{equation*}
\begin{equation*}
	q(1,j;0,j)=\mu(N+1)+\xi;
\end{equation*}
\begin{equation}
	q(k,j;0,j)=\xi,\qquad k=1,2,\ldots,N;
	\label{eq:tassi4}
\end{equation}
$$
q(k,j;r,j)=0\quad \hbox{if $|k-r|>1$}, \qquad q(k,i;r,j)=0\quad \hbox{if $i\neq j$}.
$$ 
Hence, denoting by ${\bf I}$ the identity matrix and by ${\bf 0}$ the zero matrix, the generator $Q$ of the process can 
be written as 
$$
Q=
\begin{pmatrix}
  -{\bf L}_0 & {\bf L}_0 \cdot {\bf C} & {\bf 0} & {\bf 0} & \cdots  & {\bf 0} & {\bf 0} \\
  {\bf M}_1+{\bf \Xi} & -{\bf M}_1-{\bf \Xi}-{\bf L}_1 & {\bf L}_1 & {\bf 0}  & \cdots  & {\bf 0} & {\bf 0} \\
  {\bf \Xi} & {\bf M}_2 &  -{\bf M}_2-{\bf \Xi}-{\bf L}_2 & {\bf L}_2 & \cdots  & {\bf 0} & {\bf 0} \\
  \vdots & \vdots       & \vdots				& \vdots         & \cdots  &               \vdots   & \vdots \\
 {\bf \Xi} & {\bf 0}    & {\bf 0}                    & \cdots  &             &  {\bf M}_N &  -{\bf M}_N-{\bf \Xi} \\
\end{pmatrix},
$$

where ${\bf C}$ has been defined in (\ref{c_lj}) and we have set 
$$
{\bf M}_i=\mu (N+i) {\bf I},\,\, i=1,\ldots, N, \qquad 
{\bf L}_i=\lambda (N-i) {\bf I},\,\,  i=0,\ldots, N,\qquad 
{\bf \Xi}=\xi {\bf I}.
$$
Note that $Q$ is a block matrix, with ${\bf M}_i$, ${\bf L}_i$ and ${\bf \Xi}$ diagonal matrices of order $d$, 
and, due to the presence of catastrophes, the first block column is completely occupied.

Let us set the initial state of the process as ${\bf X}_0=(0,l_0)$, with $l_0 \in D$, and define the state probabilities of ${\bf X}_t$, 
for $k=0,1,\ldots,N$, $j\in D$, as 
\begin{equation}
p(k,j, \cdot )={\mathbb P}\{({\cal N}(\cdot),J(\cdot))=(k, j)\,|\, ({\cal N}(0),J(0))=(0,l_0)\}. 
\label{eq:probpkjt}
\end{equation}
The initial conditions are thus expressed as 
\begin{equation}
	p(0,l,0)=\delta_{l,l_0},
	\label{probiniz1}
\end{equation}
where $\delta_{l,l_0}$ is the Kronecker delta, and for $k=1,\ldots, N$, 
\begin{equation}
	p(k,l,0)=0.
	\label{probiniz2}
\end{equation}
Due to Eqs. \!(\ref{eq:tassi1}) $\div$ (\ref{eq:tassi4}), for $k=1,\ldots,N, j\in D$ and $t>0$, the state probabilities $p(k,j, t)$
satisfy the following Chapman–Kolmogorov forward differential-difference equations
\begin{eqnarray}
	\label{eq:system}
	&&  \hspace{-0.8cm}
	{d \over d t}\;p(0,j, t)=\mu (N+1)\,p(1,j,t)+\xi \sum_{k=1}^Np(k,j,t)-\lambda N\,p(0,j,t), \nonumber
	\\
	&&  \hspace{-0.8cm}
	{d \over d t}\;p(1,j, t)=\mu (N+2)\,p(2,j,t)+\sum_{l \in D}c_{l,j} \lambda N p(0,l,t)-[\lambda (N-1)+\mu (N+1)+\xi]\,p(1,j,t),\nonumber
	\\
	&&  \hspace{-0.8cm}
	{d \over d t}\;p(k,j, t)=\mu (N+k+1)\,p(k+1,j,t)+\lambda (N-k+1) p(k-1,j,t)
	\\
	&& \hspace{1.5cm}-[\lambda (N-k)+\mu (N+k)+\xi]\,p(k,j,t),\qquad k \in \left\{2,\ldots,N-1\right\}\nonumber
	\\
	&&  \hspace{-0.8cm}
	{d \over d t}\;p(N,j, t)=\lambda\,p(N-1,j, t)-[\mu2N+\xi]\,p(N,j, t).
	\nonumber
\end{eqnarray}
Hence, denoting by 
\begin{equation}
p(k,\cdot):={\mathbb P} ({\cal N}(\cdot)=k)= \sum_{j\in D} p(k,j,\cdot),\quad k\in \{0,1,\ldots,N\},
\label{pkt}
\end{equation}
the marginal probability, i.e. the probability that at time $t$ the level of the system is $k$ whatever the phase, the following system of differential-difference equations holds:
\begin{eqnarray}
	&&  \hspace{-0.8cm}
	{d \over d t}\;p(0, t)=\mu (N+1)\,p(1,t)+\xi-[\lambda N+\xi]\,p(0,t),\nonumber
	\\
	&&  \hspace{-0.8cm}
	{d \over d t}\;p(1,t)=\mu (N+2)\,p(2,t)+\lambda N p(0,t)-[\lambda (N-1)+\mu (N+1)+\xi]\,p(1,t),\nonumber
	\\
	&&  \hspace{-0.8cm}
	{d \over d t}\;p(k, t)=\mu (N+k+1)\,p(k+1,t)+\lambda (N-k+1) p(k-1,t)
	\label{sysdiffeq}
	\\
	&& \hspace{1.5cm}-[\lambda (N-k)+\mu (N+k)+\xi]\,p(k,t),\qquad k \in \left\{2,\ldots,N-1\right\} \nonumber
	\\
	&&  \hspace{-0.8cm}
	{d \over d t}\;p(N, t)=\lambda\,p(N-1,t)-[\mu 2N+\xi]\,p(N, t).
	\nonumber
\end{eqnarray}

Note that when $\xi=0$, Eqs. (\ref{eq:system}) identify with Eqs. (6) of \cite{DiCrescenzo2022}. Moreover, the case $d=2$, by assuming that 
the $d$ lines are joined at the origin so that the states $(0,j)$ identify with a common extreme $0$, corresponds to the one-dimensional 
Ehrenfest process. 

\section{Transient regime}
%
This Section is devoted to the analysis of the process in the transient regime. In particular, by exploiting the relationship between 
${\cal N}(t)$ and the corresponding process in the absence of catastrophes, we derive several results concerning the marginal distribution 
of ${\cal N}(t)$.

For $z\in [0,1]$ and $t>0$, recalling Eqs. (\ref{eq:probpkjt}) and (\ref{pkt}), let us define the $j$-phase probability generating function 
\begin{equation}
G_j(z,t):=\sum_{k=0}^N z^k p(k,j,t),
\label{Gjpgf}
\end{equation}
and the probability generating function (pgf) of ${\cal N}(t)$ 
\begin{equation}
F(z,t):=E[z^{{\cal N}(t)}]=p(0,t)+\sum_{k=1}^N z^k p(k,t)=\sum_{j=1}^d G_j(z,t).
\label{FPgrande}
\end{equation}
\begin{proposition}
	\label{prop_F}
	The pgf (\ref{FPgrande}) satisfies the following partial differential equation, for $z\in [0,1]$ and $t>0$:
	\begin{equation}
		{\partial \over \partial t}\!F(z, t)=\xi-\mu \frac{N}{z}(1-z) p(0,t)+\left[\frac{N}{z}(1-z) (\mu-\lambda z)-\xi\right]F(z,t)
		+(1-z)(\mu+\lambda z) {\partial \over \partial z}\!F(z, t),
		\label{eq:diffF}
	\end{equation}
	with initial condition
	\begin{equation}
	F(z,0)=1,
	\label{initialconditions}
\end{equation}
and  boundary conditions
\begin{equation}
	F(1,t)=1, \qquad F(0,t)=p(0,t).
	\label{boundconditions}
\end{equation}
\end{proposition}
\begin{proof}
Eq. (\ref{eq:diffF}) immediately follows from  Eq. (\ref{sysdiffeq}) and recalling Eq. (\ref{FPgrande}). Moreover, conditions
(\ref{initialconditions}) and (\ref{boundconditions}) can be obtained by virtue of (\ref{probiniz1}) and (\ref{probiniz2}).
\end{proof}

\begin{proposition}
	For all $\lambda,\,\mu>0$, Eq.\ (\ref{eq:diffF}) with conditions (\ref{initialconditions}) and (\ref{boundconditions}),
	admits of the following solution for $z\in [0,1]$ and $t\geq 0$:
\begin{eqnarray}
	&&\hspace{-0.5cm}
	F(z,t)=\frac{e^{-\xi t} \left[\lambda(1-z)+(z \lambda+\mu)e^{t(\lambda+\mu)}\right]^N\left[\mu(z-1)+(z \lambda+\mu)e^{t(\lambda+\mu)}\right]^N}{z^N e^{2Nt(\lambda+\mu)}(\lambda+\mu)^{2N}}
	\nonumber\\
	&&\hspace{-0.5cm}
	+\frac{\xi}{z^N(\lambda+\mu)^{2N}} \int_0^t e^{-\xi(t-y)} \left[\lambda(1-z)e^{-(t-y)(\lambda+\mu)}+(z \lambda+\mu)\right]^N\left[\mu(z-1)e^{-(t-y)(\lambda+\mu)}+(z \lambda+\mu)\right]^N  {\rm d}y \nonumber\\
	&&\hspace{-0.5cm}
	+\frac{N \mu (z-1)}{z^N (\lambda+\mu)^{2N-1}}
	\\
	&&\hspace{-0.5cm}
	\times
	 \int_0^t p(0,y)e^{-(\lambda+\mu+\xi)(t-y)} \left[\lambda(1-z)e^{-(t-y)(\lambda+\mu)}+(z \lambda+\mu)\right]^N\left[\mu(z-1)e^{-(t-y)(\lambda+\mu)}+(z \lambda+\mu)\right]^{N-1} {\rm d}y   .\nonumber
     \label{F(z,t)}
\end{eqnarray}
\end{proposition}
\begin{proof}
The proof immediately follows from Proposition \ref{prop_F} by adopting the method of characteristics.
\end{proof}
Let ${\widetilde {\cal N}}(t)$ denote the corresponding process without catastrophes, as investigated in \cite{DiCrescenzo2022}. 
A well-established result provides a relationship between the probabilities and conditional moments of ${\widetilde  {\cal N}}_t$ and  $ {\cal N}_t$ 
(see, for instance, \cite{DiCrescenzo2008}). Specifically, if we denote by ${\widetilde p}(k,t)$ the state probabilities of the process ${\widetilde{\cal N}}_t$, and recalling Eq. (\ref{pkt}), we obtain
\begin{equation}
  p(k,t)=e^{-\xi t}{\widetilde p}(k,t)+\xi \int_0^t e^{-\xi \tau}{\widetilde p}(k,\tau)d\tau,\;t \geq 0,\;k \in \{1,\ldots,N\}.
  \label{relazione1}
\end{equation}
The first term in the right-hand-side of Eq. (\ref{relazione1}) corresponds to the probability that at time $t$ the process occupies 
the level $k$ (in any phase) without the occurrence of a catastrophe during the interval $(0,t)$. The second
term, instead, accounts for the case in which the first catastrophe occurs at some time $\tau\leq t$. 

Similarly, denoting by ${\widetilde M}^r(t):=E\left[{\widetilde {\cal N}}_t^r\right]$ the $r$-th moment of ${\widetilde {\cal N}}(t)$, $r\geq 1$, the following relationship holds:
\begin{equation}
M^r(t):=E\left[{ {\cal N}}_t^r\right]=e^{-\xi t}{\widetilde M}^r(t)+\xi \int_0^t e^{-\xi \tau}{\widetilde M}^r(\tau)d\tau,\;t \geq 0.
  \label{rel2}
\end{equation}
In the following subsection we make use of the above equations to disclose the explicit expressions both of the marginal probabilities $p(k,t)$ and of the moments of ${\cal N}_t$, 
in the special case $\lambda=\mu$. 

\subsection{The special case: $\lambda=\mu$}

Let us assume $\lambda=\mu>0$ and set
\begin{equation}
 R(\alpha_k)=\frac {\prod_{r=0}^{N-1}\left[\alpha_k+2\mu (2r+1)\right]} {2 \prod_{\stackrel{s=1}{s\neq k}}^{N+1}  (\alpha_k-\alpha_s),},
\qquad k=1,2,\ldots, N+1,
\label{Ralfak}
\end{equation}
with $0=\alpha_1>\alpha_2>\ldots >\alpha_{N+1}$ denoting the roots of the polynomial 
\begin{equation}
\label{pol_P}
P(x)=x \left[\prod_{r=0}^{N-1} \left(x+2 \mu (2r+1)\right)+\prod_{r=0}^{N-1} \left(x+2 \mu (2r+2)\right)\right].
\end{equation}
Note that the proof that the polynomial $P(x)$ admits one root equal to $0$ and $N$ different negative roots can be found in \cite{DiCrescenzo2022}.
%
%
%
\begin{proposition}
If $\lambda=\mu$, for all $t>0$, recalling Eq. (\ref{Ralfak}), we have  
%
\begin{equation}
  p(0,t)=\frac{2 {2N \choose N}}{{2N \choose N}+4^N}+ 2 \sum_{k=2}^{N+1} R(\alpha_k)\left[e^{-(\xi-\alpha_k) t}+\frac{\xi}{\xi-\alpha_k}(1-e^{-(\xi-\alpha_k) t})\right].
  \label{rel1}
\end{equation}
%
Moreover, for $r=1,2,\ldots,N$, it results

\begin{eqnarray*}
&& \hspace*{-1.2cm}
p(r,t)=\frac{1}{4^N}{2N \choose N+r}\sum_{l=0}^N {N \choose l} (-1)^l \theta_{1}\big(2l,0,0\big) \big[1-\theta_2\big(t,4 l \mu \big)\big]
\!+\! \frac{\mu N}{2^{2N-2}}(-1)^{N-1} 
\nonumber \\
&& \hspace*{-0.2cm}
\times 
\sum_{j=0}^{N-1}{N-1 \choose j}(-1)^{N-1-j}
 \left\{{2N-1 \choose N+r} \theta_{1}\big(2j,1,1\big) -{2N-1 \choose N+r-1}\theta_{1}\big(2j,0,1\big)\right\}
\nonumber\\
&& \hspace*{-0.2cm}
\times \sum_{k=1}^{N+1}\frac{R(\alpha_k)}{\left|\alpha_k\right|-2\mu(2N-2j-1)}\left\{\theta_2\big(t,2 \mu [2 N-2 j-1]\big)-\theta_2\big(t,|\alpha_k|\big)\right\}
\nonumber\\
&& \hspace*{-0.2cm}
+\frac{\mu N}{2^{2N-2}}(-1)^{N-r}{2N \choose N+r}\sum_{j=0}^{N-1}{N-1 \choose j}(-1)^{N-1-j}\theta_{1}\big(2j,0,0\big) 
\nonumber\\
&& \hspace*{-0.2cm}
\times \sum_{k=1}^{N+1}\frac{R(\alpha_k)}{|\alpha_k|-4\mu(N-j)}
\left\{\theta_2\big(t,4 \mu (N- j)\big)-\theta_2\big(t,|\alpha_k|\big)\right\},
\end{eqnarray*}
where we have set $\theta_1(a,b,c):={}_{2}F_{1}\left(-a,-N+r+b,-2N+c,2\right)$ and 
$\theta_2(t,y)\!:=\!(1-{\rm e}^{-t(\xi+y)})\left(1-\frac{\xi}{\xi+y}\right).$
\label{proppktransiente}
\end{proposition}
\begin{proof}
The results follows from Eq.  (\ref{rel1}) and  recalling Propositions $4$ and $5$ of \cite{DiCrescenzo2022}.
\end{proof}

\begin{figure}[t]
\centering
\includegraphics[width=6cm]{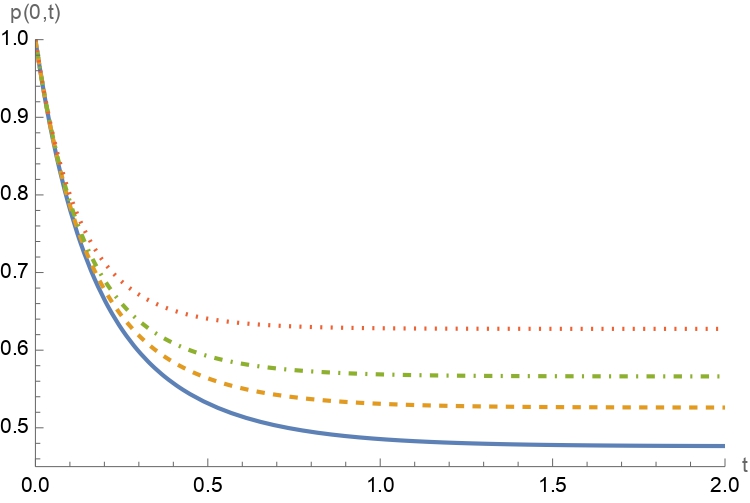}	
\hspace*{0.4cm}
\includegraphics[width=6cm]{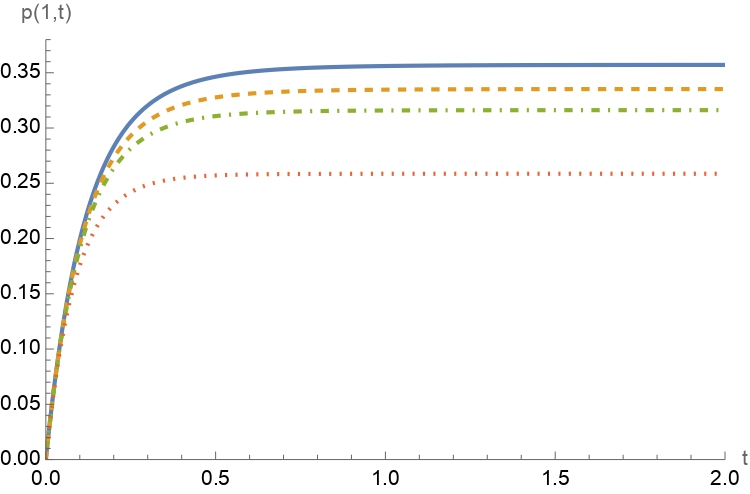}	
\includegraphics[width=6cm]{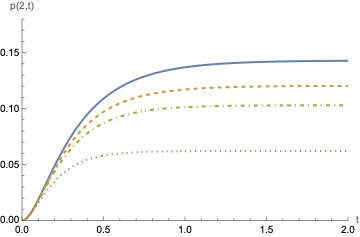}	
\hspace*{0.4cm}
\includegraphics[width=6cm]{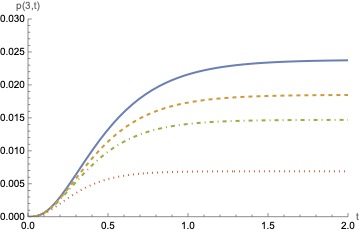}	
\caption{The marginal probabilities $p(k,t)$ as function of $t$ for $N=3$, $\lambda=\mu=1$ and $\xi=0$ (full line), $\xi=0.5$ (dashed line), $\xi=1$ (dot-dashed line), $\xi=2$ (dotted line).}
\label{Figure_trans_N_3_prob}
\end{figure}

Some plots of the probabilities $p(k,t)$ as function of $t$ for different values of $\xi$ are provided in Figure \ref{Figure_trans_N_3_prob}. 
The probability $p(0,t)$ is decreasing with respect to $t$ and increasing as $\xi$ grows (indeed, the more frequently catastrophes occur, the greater the likelihood that the system contains zero elements). 
Conversely, for $k=1,2,3$, the probabilities $p(k,t)$ are increasing in $t$ and decreasing as $\xi$ increases.

%

%
%
%
%
\begin{proposition}
For $\lambda=\mu>0$, the first and the second moment of the process ${\cal N}_t$ are given by
\begin{eqnarray}
&& \hspace*{-1cm}
M(t)=2 \mu N \xi \sum_{k=1}^{N+1}  \frac{R(\alpha_k)}{|\alpha_k|-2 \mu }\left[\frac{1}{2\mu+\xi}-\frac{1}{\xi+|\alpha_k|}\right]\nonumber\\
&&\hspace{1cm}+2 \mu N \sum_{k=1}^{N+1}  \frac{R(\alpha_k)}{|\alpha_k|-2 \mu }\left[\frac{2\mu}{2\mu+\xi}e^{-t(\xi+2\mu)}-\frac{|\alpha_k|}{|\alpha_k|+\xi}e^{-t(\xi+|\alpha_k|)}\right],
\nonumber\\
&& \hspace*{-1cm}
M^2(t)=
\frac{2 N \mu}{4 \mu+\xi}\left(1-e^{-(\xi+4\mu) t}\right)
-2 N \mu  \left[\sum_{k=1}^{N+1}\frac{\xi\,R(\alpha_k)}{(\xi+4 \mu)(|\alpha_k|+\xi)}\right.
\nonumber\\
&& \left.\hspace*{-0cm}+ \sum_{k=1}^{N+1}\frac{R(\alpha_k)}{(|\alpha_k|-4 \mu)}\left( \frac{4\mu }{4 \mu+\xi}e^{-t(\xi+4 \mu)}-\frac{|\alpha_k|}{|\alpha_k|+\xi}e^{-t(\xi+|\alpha_k|)}\right)\right]\nonumber,
\end{eqnarray}
where $R(\alpha_k)$ and $\alpha_k$ have been defined in Eqs.  (\ref{Ralfak}) and (\ref{pol_P}), respectively.
\end{proposition}
\begin{proof}
Starting from the expression of the probability generating function in \cite{DiCrescenzo2022}, we can easily obtain 
the expected value $\widetilde{M}(t)$ and the second-order moment $\widetilde{M}^2(t)$  of the process without catastrophes ${\widetilde {\cal N}_t}$
\begin{eqnarray}
&&\widetilde{M}(t)=2 \mu N \sum_{k=1}^{N+1}  \frac{R(\alpha_k)}{|\alpha_k|-2 \mu }
({\rm e}^{-2 \mu  t}-e^{-|\alpha_k|\, t} ),\nonumber\\
&&\widetilde{M}^2(t)=\frac{N}{2}\left(1-e^{-4\mu t}\right)-2 \mu N \sum_{k=1}^{N+1}  \frac{R(\alpha_k)}{|\alpha_k|-4 \mu }
({\rm e}^{-4 \mu  t}-e^{-|\alpha_k|\, t} ).\nonumber
\end{eqnarray}
Then, the result follows from Eq. (\ref{rel2}) after straightforward calculations.
\end{proof}
%
\begin{figure}[t]
\centering
\includegraphics[width=6cm]{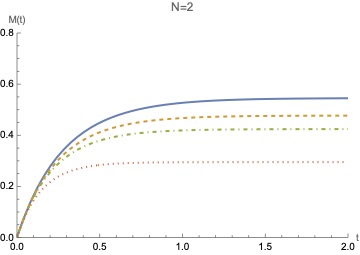}
\hspace*{0.4cm}
\includegraphics[width=6cm]{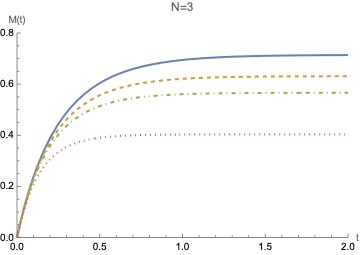}	
\caption{The expected value $M(t)$ as function of $t$ for $N=2$ on the left and $N=3$ on the right, $\lambda=\mu=1$ and $\xi=0$ (full line), $\xi=0.5$ (dashed line), $\xi=1$ (dot-dashed line), $\xi=2$ (dotted line).}
\label{trans_N_2_media}
\end{figure}
\begin{figure}[t]
\centering
\includegraphics[width=6cm]{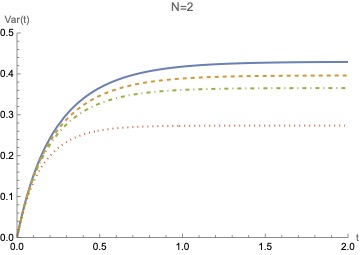}
\hspace*{0.4cm}
\includegraphics[width=6cm]{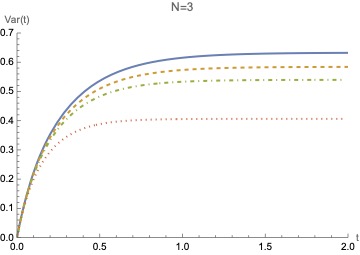}	
\caption{The variance $Var(t)=M^2(t)-[M(t)]^2$ is plotted as a function of $t$ for $N=2$ on the left and $N=3$ on the right, $\lambda=\mu=1$ and $\xi=0$ (full line), $\xi=0.5$ (dashed line), $\xi=1$ (dot-dashed line), $\xi=2$ (dotted line).}
\label{trans_N_2_var}
\end{figure}

Figures \ref{trans_N_2_media} and \ref{trans_N_2_var} show the expected value and the variance of ${\cal N}(t)$ for some choices of $N$ and different values of $\xi$. 
Both functions exhibit an increasing behaviour with respect to time  $t$, while they decrease as the parameter $\xi$ increases.
\section{Asymptotic regime}
This section is devoted to the analysis of the process ${\bf X}_t$ in its asymptotic regime.
In particular, we establish several results concerning both the state probabilities and the marginal distribution of the process.

Let us set, for $ j=1,\ldots,d$, 

\begin{equation}
  G_j(z):=\lim_{t \rightarrow +\infty}G_j(z,t),
  \label{G_j}
\end{equation}
and
\begin{equation*}
	F(z):=\lim_{t\to +\infty} F(z,t),
\end{equation*}
with $G_j(z,t)$ and $F(z,t)$ defined in Eqs. (\ref{Gjpgf}) and (\ref{FPgrande}), respectively.

Moreover, denoting by ${\cal N}$ and ${\cal J}$ the random variables describing the stationary state of ${\cal N}(t)$ and ${\cal J}(t)$ respectively, we set, 
for $k=0,1,\ldots,N$ and $j=1.\ldots,d$,

\begin{equation}
\rho(k,j):=\lim_{t \rightarrow +\infty} p(k,j,t)=P({\cal N}=k, {\cal J}=j),
    \label{asympt_prob_bid}
\end{equation}
and 
\begin{equation}
	\rho(k):=\lim_{t\rightarrow +\infty} p(k,t)=P({\cal N}=k), \qquad 
	\label{rhok}
\end{equation}
with 
\begin{equation}
    \rho(k)=\sum_{j=1}^d \rho(k,j).
    \label{rel_asympt_prob_bid}
\end{equation}

Note that the following balance equation between the incoming and outgoing probability currents through level $0$ holds:

\begin{equation}
\sum_{l \neq j} c_{j,l} \rho(0,j)= \sum_{l \neq j}c_{l,j}  \rho(0,l),
\label{balance}
\end{equation}
with $c_{i,j}$ defined in Eq. (\ref{c_lj}). 
In such equation, the left-hand-side represents the intensity of transitions leaving the phase $j$ at level $0$ 
moving toward the level $1$ of any different phase, whereas the right-hand-side represents the intensity of transitions entering 
the phase $j$ of the level $1$, starting from any different phase of the level $0$.

The following Proposition provides the explicit expression of $F(z)$.
\begin{proposition}\label{asymptotic}
For $	\lambda, \mu, \xi>0$, $d, N\in {\mathbb N}^+$, the asymptotic probability generating function $F(z)$ is given by
\begin{eqnarray}
	&&\hspace{-0.6cm}
	F(z)=-\frac{N \mu (1-z)(\lambda z+\mu)^N}{(\lambda+\mu)^{2N}z^N} g(\lambda,\mu,\xi,N)\sum_{j=0}^{N-1}{N-1 \choose j} [(\lambda+\mu)z]^{N-j-1}[\mu(1-z)]^j 
	\nonumber\\ 
	&&\hspace{-0.6cm}
	\times \frac{\Gamma\left(j+1\right)\Gamma\left(\frac{\xi}{\lambda+\mu}+1\right)}{\Gamma\left(j+\frac{\xi}{\lambda+\mu}+2\right)}{}_{2}F_{1}\left(\frac{\xi}{\lambda+\mu}+1,-N,\frac{\xi}{\lambda+\mu}+j+2;-\frac{\lambda(1-z)}{\lambda z +\mu}\right) 
	+\frac{ (\lambda z+\mu)^N}{(\lambda+\mu)^{2N}z^N}  \sum_{j=0}^{N}{N \choose j} 
	\nonumber\\ 
	&&\hspace{-0.6cm}
		\times [(\lambda+\mu)z]^{N-j}[\mu(1-z)]^j \frac{\Gamma\left(j+1\right)\Gamma\left(\frac{\xi}{\lambda+\mu}+1\right)}{\Gamma\left(j+\frac{\xi}{\lambda+\mu}+1\right)}
	{}_{2}F_{1}\left(\frac{\xi}{\lambda+\mu},-N,\frac{\xi}{\lambda+\mu}+j+1;-\frac{\lambda(1-z)}{\lambda z +\mu}\right), 
	\nonumber\\
	&&\hspace{-1cm}
	\label{F_asympt}
\end{eqnarray}					
	where
	\begin{equation}
		g(\lambda,\mu,\xi,N):=\frac{{}_{2}F_{1}\left(\frac{\xi}{\lambda+\mu},-N,\frac{\xi}{\lambda+\mu}+N+1;-\frac{\lambda}{\mu}\right)}{{}_{2}F_{1}\left(\frac{\xi}{\lambda+\mu}+1,-N,\frac{\xi}{\lambda+\mu}+N+1;-\frac{\lambda}{\mu}\right)},
		\label{g_1}
	\end{equation}
	with ${}_{2}F_{1}\left(a,b,c;z\right)$ denoting the Gauss hypergeometric function.
\end{proposition}
The proof is provided in  \ref{proof_prop_asym}.
\par
The result obtained in Proposition \ref{asymptotic} allows us to disclose the explicit expression of the probabilities $\rho(k)$, for all $\lambda,\mu,\xi>0$.

%
\begin{proposition}
For $r=1,\ldots,N$ the asymptotic probabilities (\ref{rel_asympt_prob_bid}) are given by
\begin{eqnarray}
&& \hspace*{-1.8cm}
\rho(r)=\Gamma\left(\frac{\xi}{\lambda+\mu}+1\right)\left(\frac{\lambda}{\lambda+\mu}\right)^r {N \choose r}\left[g(\lambda,\mu,\xi,N)\left(\frac{\mu}{\lambda+\mu}\right)\right.
\nonumber\\
&& \hspace*{-1.2cm}
+\sum_{j=0}^{n-1}\left(-\frac{\mu}{\lambda+\mu}\right)^j\frac{n!(n-j)}{j!(n-j)!(j+1+r)}\frac{\Gamma(j+2+r)}{\Gamma\left(2+j+r+\frac{\xi}{\lambda+\mu}\right)}
\nonumber\\
&& \hspace*{-1.2cm}
\times\, \,  {}_{3}F_{2}\left(\{j+1+r,-N+r,2+j+r\},\left\{1+r,2+j+r+\frac{\xi}{\lambda+\mu}\right\},\frac{\lambda}{\lambda+\mu}\right)
\nonumber\\
&& \hspace*{-1.2cm}
+\sum_{j=0}^{n}\left(-\frac{\mu}{\lambda+\mu}\right)^j{n \choose j}\frac{\Gamma(j+1+r)}{\Gamma\left(1+j+r+\frac{\xi}{\lambda+\mu}\right)}
\nonumber\\
&& \hspace*{-1.2cm}
\times  \left.{}_{3}F_{2}\left(\{j+1+r,1+j+r,-n+s\},\left\{1+r,1+j+r+\frac{\xi}{\lambda+\mu}\right\},\frac{\lambda}{\lambda+\mu}\right)\right],
\label{prob_asympt}
\end{eqnarray}
where ${}_{3}F_{2}(\{\gamma_1,\gamma_2,\gamma_3\},\{\delta_1, \delta_2\},z)$ denotes the generalized hypergeometric function.
Moreover, for $r=0$ we have
\begin{eqnarray*}
&& \hspace*{-1.2cm}
\rho(0)=\frac{\xi}{\xi+N(\lambda+\mu)}{}_{3}F_{2}\left(\left\{\frac{1}{2},-N,N+1\right\},\left\{1-N-\frac{\xi}{\lambda+\mu},1+N+\frac{\xi}{\lambda+\mu}\right\},\frac{4\lambda \mu}{(\lambda+\mu)^2}\right)
\nonumber\\
&& \hspace*{-0.6cm}
-g(\lambda,\mu,\xi,N) \Gamma\left(\frac{\xi}{\lambda+\mu}+1\right) \sum_{j=1}^N {N \choose j}\left(-\frac{\mu}{\lambda+\mu}\right)^j
\frac{\Gamma(j+1)}{\Gamma\left(j+1+\frac{\xi}{\lambda+\mu}\right)}
\nonumber \\
&& \hspace*{-0.6cm}
\times \,  {}_{3}F_{2}\left(\{j,1+j,-N\},\left\{1,1+j+\frac{\xi}{\lambda+\mu}\right\},\frac{\lambda}{\lambda+\mu}\right),
 \end{eqnarray*}
with $g(\lambda,\mu,N)$  defined in Eq. (\ref{g_1}).
\label{probasintotiche}
\end{proposition}
\begin{proof}
The results are obtained through the series expansion of the asymptotic probability generating function given in Eq. (\ref{F_asympt}).
\end{proof}
The following Proposition provides the expression of the asymptotic probabilities (\ref{rhok}) in the special case $\lambda=\mu$.
%
\begin{proposition}\label{prob_asympt_l_equal_m}
When $\lambda=\mu>0$ the asymptotic probability $\rho(0)$ is given by
\begin{equation}
\rho(0)=\frac{2}{1+\frac{\left(1+\frac{\xi}{4 \lambda}\right)_N}{\left(\frac{1}{4}\left(2+\frac{\xi}{\lambda}\right)\right)_N}}.
\label{prob_asympt_0_l_equal_m}
\end{equation}
Moreover, for $r=1,\ldots,N$, we have
\begin{eqnarray}
&& \hspace*{-1.2cm}
\rho(r)=\left(\frac{1}{4}\right)^N {2N \choose N-r} r! \frac{\Gamma\left(1+\frac{\xi}{4 \lambda}\right)\Gamma\left(\frac{1}{2}+\frac{\xi}{4 \lambda}\right)}{\Gamma\left(1+N+\frac{\xi}{4 \lambda}\right)\Gamma\left(\frac{1}{2}+\frac{\xi}{4 \lambda}\right)+\Gamma\left(1+\frac{\xi}{4 \lambda}\right)\Gamma\left(N+\frac{1}{2}+\frac{\xi}{4 \lambda}\right)}
\nonumber \\
&& \hspace*{-0.5cm}
\times \left({}_{3}F_{2}\left(\{r-N,\frac{1}{2}+r,\frac{\xi}{4\lambda}\},\left\{\frac{1}{2}-N,r+1+\frac{\xi}{4\lambda}\right\},1\right)\frac{\Gamma\left(1+N+\frac{\xi}{4 \lambda}\right)}{\Gamma\left(1+r+\frac{\xi}{4 \lambda}\right)}\right.
\nonumber\\
&& \hspace*{-0.5cm}
+\left.{}_{3}F_{2}\left(\{r-N,\frac{1}{2}+r,\frac{\xi}{4\lambda}+\frac{1}{2}\},\left\{\frac{1}{2}-N,r+\frac{1}{2}+\frac{\xi}{4\lambda}\right\},1\right)\frac{\Gamma\left(\frac{1}{2}+N+\frac{\xi}{4 \lambda}\right)}{\Gamma\left(\frac{1}{2}+r+\frac{\xi}{4 \lambda}\right)}\right).
\label{prob_asympt_l_equal_mu}
\end{eqnarray}
\end{proposition}
\begin{proof}
The proof is provided in  \ref{proof_prob_asympt_l_equal_m}.
\end{proof}
\begin{figure}[t]
\centering
\includegraphics[width=5cm]{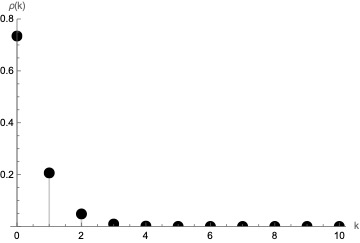}
\includegraphics[width=5cm]{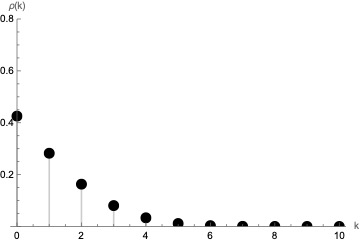}
\includegraphics[width=5cm]{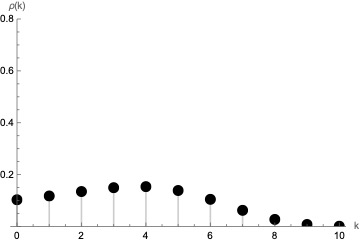}
\caption{The probabilities $\rho(k)$ given in (\ref{prob_asympt}) are plotted for $N=10$, $\lambda=1,\mu=3$, $\lambda=\mu=1$, $\lambda=3,\mu=1$ from left to right, with  $\xi=2$.
}
\label{prob_asympt_N10}
\end{figure}
 
Figure \ref{prob_asympt_N10} shows the probabilities $\rho(k)$, for $N=10$, $\xi=2$ and different values of $\lambda$ and $\mu$. When $\lambda \leq \mu$ the probability $\rho(k)$ is decreasing in $k$. Conversely, when $\lambda>\mu$ the function $\rho(k)$ increases up to a certain value of $k$ and then decreases thereafter.

\begin{remark}
When $\lambda = \mu$ and $d=2$ it is possible to establish an explicit correspondence between the stationary probabilities  $\rho(r)$ given in Eq. (\ref{prob_asympt_l_equal_mu})
and those of the classical Ehrenfest model, denoted by $q_r$, as reported in Eq. (18) of \cite{Dharmaraja}.
To derive this relationship, we first introduce a suitable normalization constant $c(\lambda, \xi,N)$ defined as
$$
c(\lambda, \xi,N)=\left(\sum_{j=-N}^0 q_j+\sum_{j=0}^N q_j  \right).
$$
Then, the following identity holds:
$$
\rho(r)=\frac{q_r+q_{-r}}{c(\lambda, \xi,N)},  \qquad r=0,1,\ldots,N.
$$
It is worth noting that, due to the special role of state $0$ in the present multi-type Ehrenfest model, one obtains
$\rho(0) = 2 q_0/c(\lambda, \xi,N)$.
\end{remark}


\begin{figure}[t]
\centering
\includegraphics[width=7cm]{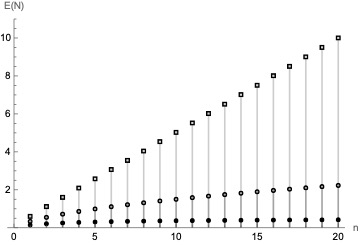}
\includegraphics[width=7cm]{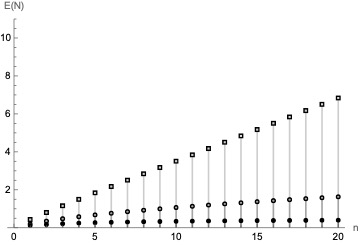}
\caption{The asymptotic expected value without catastrophes ($\xi=0$) on the left and with catastrophes ($\xi=2$) on the right, for $N =0,1,\ldots,20$ with $\lambda=\mu=1$ (empty circle), $\lambda=1,\mu=3$ (full circle), $\lambda=3,\mu=1$ (square).}
\label{asympt_mean}
\end{figure}

\begin{figure}[t]
\centering
\includegraphics[width=7cm]{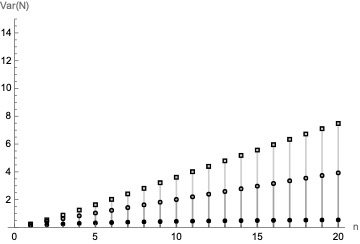}
\includegraphics[width=7cm]{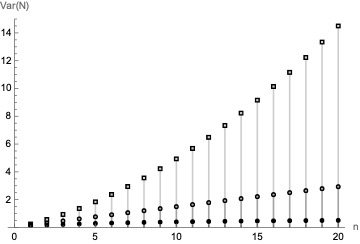}
\caption{The asymptotic variance without catastrophes ($\xi=0$) on the left and with catastrophes ($\xi=2$) on the right, for $N =0,1,\ldots,20$ with $\lambda=\mu=1$ (empty circle), $\lambda=1,\mu=3$ (full circle), $\lambda=3,\mu=1$ (square).}
\label{asympt_var}
\end{figure}		

The stationary probability $\rho(0)$ plays a relevant role, as it represents the asymptotic probability that the system is empty. 
For this reason, in the sequel we focus on its characterization and analysis, when $\lambda=\mu$.

\begin{proposition}
If $\lambda=\mu$, the asymptotic probability $\rho(0)$ in (\ref{prob_asympt_0_l_equal_m}) exhibits a monotonic dependence on the parameters: it is a decreasing function with respect to both $\lambda$ and $N$, 
while it increases as a function of $\xi$.
\end{proposition}
\begin{proof}
Let us set $\eta:=\frac{\xi}{4\,\lambda}$, so that Eq. (\ref{prob_asympt_0_l_equal_m}) can be rewritten as 
\begin{equation*}
	\rho^{\eta}_N(0):=\frac{2}{1+\frac{\left(1+\eta \right)_N}{\left(\frac{1}{2}+\eta \right)_N}}.
\end{equation*}
The asymptotic behavior of the probability $\rho^{\eta}_N(0)$ is rigorously investigated by computing its derivative with respect to a parameter. The sign of
$\frac{\partial}{\partial \eta}\rho^{\eta}_N(0)$ is governed by the sign of the following term 
\begin{eqnarray}
&& \hspace*{-1cm}
\left(
 \psi^{(0)}\left(\frac{1}{2} + N + \eta\right) -\psi^{(0)}\left(\frac{1}{2} + \eta\right) -
 \psi^{(0)}(1 + N + \eta)+\psi^{(0)}(1 + \eta)\right) {\rm d}\eta
 \nonumber
 \\
 && \hspace*{-1cm}
=\left(\frac{1}{2}\sum_{k=0}^{N-1}\frac{1}{(1+\eta+k)\left(\frac{1}{2}+\eta+k\right)}\right) {\rm d}\eta,
 \label{sign_der_rho0_eta}
\end{eqnarray}
where $\psi^{(n)}(z)$ is the Polygamma function defined in Eq. 6.4.1 of \cite{Abramowitz}, and the last equality is due to Eq. 6.4.10 of \cite{Abramowitz}. 
Given that the quantity within the brackets of (\ref{sign_der_rho0_eta}) is a summation of positive terms, it is inherently positive. Consequently, the sign of $\frac{\partial}{\partial \eta}\rho^{\eta}_N(0)$ is determined 
by the sign of the differential ${\rm d}\eta = {\rm d}\left(\frac{\xi}{4\lambda}\right)$. This leads to two critical observations regarding the component parameters: (i) 
viewing $\rho(0)$ as a function of $\xi$ (and holding $\lambda$ and $N$ constant), the sign of the derivative is positive, confirming that $\rho(0)$ is an increasing function of $\xi$; 
(ii) when considering $\rho(0)$ as a function of $\lambda$ (with $\xi$ and $N$ constant), the sign of the derivative is negative, establishing that $\rho(0)$ is a decreasing function of $\lambda$.
\par
Regarding the monotonicity with respect to the parameter $N$, we evaluate the derivative of $\rho^{\eta}_N(0)$ with respect to $N$: its sign is determined by the term  
$\psi^{(0)}\left(\frac{1}{2}+ N + \eta\right)-\psi^{(0)}\left(1+ N + \eta\right)$. Since the Polygamma function $\psi^{(n)}(z)$ is an increasing function of $z$,  we have that 
 $\rho(0)$ is a decreasing function of $N$.
 \end{proof}

The next proposition provides the asymptotic mean and variance of ${\cal N}$.

\begin{proposition}\label{mean_var_asym}
 The expected value and variance of ${\cal N}$ are given by
\begin{eqnarray}
&& \hspace*{-0.7cm}
   E[{\cal N}]=\frac{N}{\lambda+\mu+\xi}\left[(\lambda+\mu(g(\lambda,\mu,N)-1)\right],
   \nonumber\\
  && \hspace*{-0.7cm}
    Var[{\cal N}]=\frac{N}{(\lambda+\mu+\xi)^2}\left\{\frac{2 \lambda \mu (\lambda+\mu) [2-(n+1)g(\lambda,\mu,N)]}{2(\lambda+\mu)+\xi}\right.
     +N \mu^2 g(\lambda,\mu,N)\left[\frac{2(\lambda+\mu)}{2(\lambda+\mu)+\xi}-g(\lambda,\mu,N)\right]
    \nonumber\\
    && \hspace*{-0.7cm}
    \left.+\frac{\xi}{2(\lambda+\mu)+\xi}[(1+n)(\lambda^2+\mu^2)-2 \lambda \mu(n-3)-\mu (3\lambda+\mu)g(\lambda,\mu,N)+\xi(\lambda+\mu(1-g(\lambda,\mu,N)))]\right\}, \nonumber
    \end{eqnarray}
    where $g(\lambda,\mu,\xi,N)$ has been defined in Eq. (\ref{g_1}).
    \end{proposition}
\begin{proof}
The proof is obtained by direct calculations starting from the expression of the probability generating function given in Eq. (\ref{F_asympt}). 
\end{proof}

Subsequently, we focus our analysis on the asymptotic behavior of the mean and variance under the condition that $N$ tends to infinity and $\lambda<\mu$. 
This corresponds to evaluating the average number and the variance of system in the asymptotic regime assuming that 
$N$ is sufficiently large.
\begin{proposition}\label{Prop_lim_mean_var}
If $\lambda<\mu$, the asymptotic mean, given in Proposition \ref{mean_var_asym}, admits the following behavior as $N \rightarrow +\infty$ 
$$
\lim_{N \rightarrow +\infty} E[{\cal N}]=\frac{\lambda}{\mu-\lambda}.
$$
\end{proposition}
\begin{proof}
The proof follows from Proposition \ref{mean_var_asym}, by applying the approximation provided in Eq. (\ref{g_approx}) of Lemma \ref{lemma_appross_g}.
\end{proof}
For the case $\lambda < \mu$, the explicit expression of the asymptotic variance as $N \rightarrow +\infty$ has been derived in \ref{app_lim_var}.
Figures \ref{asympt_mean} and \ref{asympt_var} illustrate the expected value and variance of ${\cal N}$ for some values of $\lambda, \mu, \xi$. 
It should be observed that both quantities diverge to $+\infty$ when $\lambda\geq \mu$.

The following result establishes a relationship that represents the asymptotic probabilities in (\ref{asympt_prob_bid}) directly in terms of the marginal 
asymptotic probabilities given in (\ref{rhok}).

 \begin{theorem}
For $\lambda=\mu$, $j \in \{1,\ldots,d\}$, $r=0,\ldots,N$,
\begin{equation}
\rho(r,j)=\rho(r)\cdot \pi_j,
\label{independence}
\end{equation}
where  ${\vec \pi}=(\pi_1,\ldots,\pi_d)$ is the invariant law of the stochastic matrix $C$.
\end{theorem}
\begin{proof}
Since Eq. (\ref{balance}) can be written as 
\begin{equation}
\rho(0,j)=\sum_{l=1}^d c_{l,j} \rho(0,l),
\label{balance2}
\end{equation}
we obtain that the asymptotic probability generating function $G_j(z)$, defined in (\ref{G_j}), satisfies the following differential equation
\begin{equation*}
\lambda (1-z)(1+z)\frac{d}{dz}G_j(z)+\left[\frac{N}{z}(1-z)(\lambda-\lambda z)-\xi\right]G_j(z)+\xi \omega_j-\lambda \frac{N}{z}(1-z) \rho(0,j)=0,
\end{equation*}
with $\omega_j:= \sum_{k=0}^N \rho(k,j)$.
By solving this equation and using the series expansion of the solution, for all $k=0,1,\ldots,N$, we obtain
$$
\rho(k,j)=\rho(k) \cdot \omega_j.
$$
However, from Eq. (\ref{balance2}), we have that 
$$
\rho(0,j)=\rho(0)\cdot \pi_j,
$$
where$\vec \pi=(\pi_1,\ldots,\pi_d)$ is the invariant law of the stochastic matrix $C$, so that the proof immediately follows.
\end{proof}

Eq. (\ref{independence}) demonstrates that, in the asymptotic regime, the phase ${\cal J}$, within a given level $k$, becomes conditionally independent of the specific level $k$.
Consequently, the long-term behavior of the QBD process exhibits a decoupling phenomenon: the r.v.'s ${\cal N}$ and ${\cal J}$ become unconditionally independent, with the phase distribution being governed by the invariant law of the stochastic matrix $C$.

\section{The diffusion approximation}\label{section:approximation}
Hereafter, we present the construction of a diffusion approximation for the process ${\bf X}_t$. We identify the zero states of all phases with a single common state $0$, 
so that the process evolves over a star-shaped graph consisting of $d$ half lines joined at the origin.
\par
Let us consider a parameterization of the model by setting
\begin{equation}
 \lambda=\frac{\alpha}{2}+\frac{\gamma}{2}\epsilon,
 \qquad
 \mu=\frac{\alpha}{2}-\frac{\gamma}{2}\epsilon,
 \label{eq:parametri}
\end{equation}
with $\gamma>0$, $\alpha>0$ and $\epsilon>0$.
\par
For all $t>0$, let us consider the scaled process ${\cal N}^*_{\epsilon}(t)={\cal N}(t)\,\epsilon$, so that 
${{\bf X}_t^*}_{\!\epsilon}:=\{({\cal N}^*_{\epsilon}(t), J(t));\;t\geq 0\}$ 
is a continuous-time stochastic process having state space
$S_{\epsilon}^*=\{0\}\cup (\textbf{N}_{\epsilon}\times D)$, 
where $\textbf{N}_{\epsilon}=\{\epsilon,2 \epsilon,3 \epsilon,\ldots,n \epsilon\}$. 
The transient probabilities, for $t\geq 0$, $k\in \textbf{N}$, $j,\,l \in D$, are given by 
\begin{eqnarray*}
 p^*_{\epsilon}(0,l,t) \!\!  &:=& \!\!  {\mathbb P}\left\{({\cal N}^*_{\epsilon}(t),J(t))=(0,l)|({\cal N}(0),J(0))=(0,l_0) \right\}, \\
p^*_{\epsilon}(k,j,t) \!\!  &:=& \!\!  {\mathbb P}\left\{({\cal N}^*_{\epsilon}(t),J(t))=(k \epsilon,j)|({\cal N}(0),J(0))=(0,l_0) \right\}.  
\end{eqnarray*}
Being ${\cal N}^*_{\epsilon}(t)={\cal N}(t)\,\epsilon$, we have
$p^*_{\epsilon}(0,l,t)=p(0,l,t)$ and $p^*_{\epsilon}(k,j,t)=p(k,j,t)$. 
\par
In the limit as $\epsilon\to 0^+$, the scaled process ${{\bf X}_t^*}_{\!\epsilon}$ is shown to converge
weakly to a diffusion process ${\cal X}:=\{({\cal Z}(t),J(t));\;t\geq 0\}$, whose state space is the star graph
$S_{\cal X}:=\{0\}\cup\left(\mathbb{R}^+\times D\right)=\displaystyle{\cup_{j=1}^d S_j}$, where each ray $S_j$ denotes 
the positive half-line.
\par
When $\epsilon$ tends to $0$, then the above probabilities $p^*_{\epsilon}(0,l,t)$ and $p^*_{\epsilon}(k,j,t)$ tend respectively to 
$$
 {\mathbb P}\{0\leq {\cal Z}(t)<\epsilon, J(t)=l\,|\,({\cal Z}(0), J(0))=(0,l_0)\}:=f(0,l,t)\epsilon+o(\epsilon), 
$$ 
$$
 {\mathbb P}\{x\leq {\cal Z}(t)<x+\epsilon, J(t)=j\,|\,({\cal Z}(0), J(0))=(0,l_0)\}:=f(x,j,t)\epsilon+o(\epsilon),
$$ 
with $t\geq 0$, $x=k\epsilon \in \mathbb{R}^+$, and $j,l \in D$. 
Hence, $f(0,l,t)$ and $f(x,j,t)$ denote the probability density of the process ${\cal Z}(t)$ at time $t$ in the state $0$ of the $l$-ray and in 
the state $x$ of the $j$-ray, respectively.
\begin{proposition}\label{prop:diffj}
Under the limit conditions
$$
\epsilon \rightarrow 0^+,\;N \rightarrow +\infty,\;N\epsilon \rightarrow +\infty,\;N\epsilon^2 \rightarrow \nu>0,
$$
for $x\in \mathbb{R}^+$, $t\geq 0$ and $j\in D$, the density $f(x,j,t)$ satisfies the 
following partial differential equation
\begin{equation}
 {\partial\over\partial t}\;f(x,j,t)=-{\partial\over\partial x}\;
 \Bigl\{[-\alpha(x-\beta)]\,
 f(x,j,t)\Bigr\}
 +{1\over 2}\,\sigma^2{\partial^2\over\partial x^2}f(x,j,t)-\xi f(x,j,t),
 \label{eq:equdiff}
\end{equation}
with the initial condition 
$$
 f(x,j,0)=\delta(x)\delta_{j, l_0}, \qquad x\in \{0\}\cup \mathbb{R}^+, \;\; j\in D,
 $$
where $\delta(x)$ is the delta-Dirac function, and the boundary condition
\begin{equation}
  \sum_{l \in D} \left\{\left.\alpha \beta f(0,l,t)-\frac{\sigma^2}{2}\frac{\partial}{\partial x}f(x,l,t)\right|_{x=0} \right\}-\xi=0,
  \label{eq:rifless}
\end{equation}
with 
\begin{equation}
 \sigma^2=\alpha \nu , \qquad \beta=\frac{\gamma \nu}{\alpha}.
  \label{eq:parsigbet}
  \end{equation}
Moreover, the following relations hold
\begin{equation}
f(0,j,t)=\sum_{l \in D} c_{l,j} f(0,l,t),\qquad \forall j \in D,
\label{cond_approx_diff1}
\end{equation}
\begin{equation}
\lim_{x \rightarrow +\infty} f(x,j,t)=0,\qquad \forall j \in D.
\label{cond_approx_diff2}
\end{equation}
\end{proposition}
\begin{proof}
Since $p^*_{\epsilon}(k,j,t)\approx f(k \epsilon,j,t)\, \epsilon$, in analogy with Eqs. (\ref{eq:system}), for $x=k\epsilon$ with $k=2,3\ldots N-1$, $j\in D$ and $t\geq 0$ we have
\begin{eqnarray}
&&  \hspace{-0.8cm}
{\partial \over \partial t}\;\sum_{l\in D}f(0,l, t) \cdot \epsilon=\mu \frac{N_\epsilon+\epsilon}{\epsilon}\sum_{l\in D} f(\epsilon,l,t)\cdot \epsilon-\left[\lambda \frac{N_\epsilon}{\epsilon}+\frac{\xi \epsilon}{\epsilon}\right]\,\sum_{l\in D}f(0,l,t) \cdot \epsilon+\frac{\xi \epsilon}{\epsilon},\label{eq:system_diff_approx1}
\\
&&  \hspace{-0.8cm}
{\partial \over \partial t}\;f(\epsilon,j, t)\cdot \epsilon =2 \mu ({N_\epsilon+2 \epsilon})\,\frac{f(2\epsilon,j,t)}{2 \epsilon}\cdot \epsilon+ \lambda \frac{N_\epsilon}{\epsilon} 
\sum_{l \in D}c_{l,j}\, f(0,l,t)\cdot \epsilon \nonumber \\
 &&  \hspace{1.4cm}  - [\alpha N_\epsilon-(\gamma \epsilon)\epsilon +\xi \epsilon] \frac{f(\epsilon,j,t)}{\epsilon}
 \cdot \epsilon,\label{eq:system_diff_approx2}
\\
&&  \hspace{-0.8cm}
\frac{\partial}{\partial t}f(x,j,t)\cdot \epsilon=\mu \frac{N_\epsilon+x+\epsilon}{\epsilon}f(x+\epsilon,j,t)\cdot \epsilon+\lambda \frac{N_\epsilon-x+\epsilon}{\epsilon}f(x-\epsilon,j,t)\cdot \epsilon
 \nonumber\\
 &&  \hspace{1.4cm}  - \frac{(\lambda+\mu)N_\epsilon-(\lambda-\mu)x}{\epsilon}f(x,j,t)\cdot \epsilon-\frac{\xi \epsilon}{\epsilon}f(x,j,t)\cdot \epsilon,\label{eq:system_diff_approx3}
\\
&&  \hspace{-0.8cm}
{\partial \over \partial t}\;f(N_\epsilon,j, t)\cdot \epsilon=\lambda\,f(N_\epsilon-\epsilon,j, t)\cdot \epsilon- 
{\left[\left(\frac{\alpha}{2}-\frac{\gamma}{2}\epsilon\right)\frac{2 N_\epsilon}{\epsilon}+\xi\right]}
\,f(N_\epsilon,j, t) \cdot \epsilon.\label{eq:system_diff_approx4}
\end{eqnarray}
where $N_\epsilon=N\cdot \epsilon$. 
Expanding $f$ as Taylor series in Eq. (\ref{eq:system_diff_approx3}),  we obtain 
\begin{eqnarray}
&& \hspace{-1cm}
\frac{\partial}{\partial t}f(x,j,t)=(\mu+\lambda-\xi)f(x,j,t)+\left[(\mu-\lambda)N_\epsilon+(\mu+\lambda)x+(\mu-\lambda)\epsilon\right]\frac{\partial}{\partial x}f(x,j,t)
 \nonumber \\
 &&  \hspace{1.4cm}  + \frac{\epsilon}{2}\left[(\mu+\lambda)N_\epsilon+(\mu-\lambda)x+(\mu+\lambda)\epsilon\right]\frac{\partial^2}{\partial x^2}f(x,j,t)+o(\epsilon^2).
\nonumber
\end{eqnarray}
Due to (\ref{eq:parametri}), one has
\begin{eqnarray}
&& \hspace{-1cm}
\frac{\partial}{\partial t}f(x,j,t)=(\alpha-\xi) f(x,j,t)+\left[-\gamma \epsilon N_\epsilon+\alpha x-\gamma \epsilon^2\right]\frac{\partial}{\partial x}f(x,j,t)
 \nonumber \\
 &&  \hspace{1.4cm}  + \frac{\epsilon}{2}\left[\alpha N_\epsilon-\gamma \epsilon x+\alpha \epsilon\right]\frac{\partial^2}{\partial x^2}f(x,j,t)+o(\epsilon^2).
\nonumber
\end{eqnarray}
Hence, by letting 
$$
 \epsilon\rightarrow0^+, \qquad 
 N\rightarrow +\infty, \qquad 
 N \epsilon=N_\epsilon \rightarrow+\infty, \qquad 
 N \epsilon^2=N_\epsilon \epsilon \rightarrow \nu>0, 
$$
the previous partial differential equation becomes
$$
\frac{\partial}{\partial t}f(x,j,t)=(\alpha-\xi) f(x,j,t)+\left[-\gamma \nu+\alpha x\right]\frac{\partial}{\partial x}f(x,j,t)+\frac{\alpha \nu}{2} \frac{\partial^2}{\partial x^2}f(x,j,t),
$$
that coincides with (\ref{eq:equdiff}) due to (\ref{eq:parsigbet}). 
With the analogous procedure, the conditions (\ref{eq:rifless}), (\ref{cond_approx_diff1}) and (\ref{cond_approx_diff2}) follow
from Eqs. (\ref{eq:system_diff_approx1}) , (\ref{eq:system_diff_approx2}) and (\ref{eq:system_diff_approx4}).
\end{proof}
\par
Due to Proposition \ref{prop:diffj}, the diffusion approximation of ${\bf X}_t$ leads to $d$ jump diffusion processes, each defined on a ray of the star graph. 

%
\par
Let us now introduce the function
\begin{equation}
 h(x,t):=\frac{\partial}{\partial x} P({\cal Z}_t\leq x\,|\,{\cal Z}_0=0)=\sum_{j=1}^d f(x,j,t),\qquad x\in \mathbb{R}^+,\quad t\geq 0.
  \label{eq:definhxt}
\end{equation}
\begin{proposition}
For $x\in \mathbb{R}^+$ and $t\geq 0$, the transition density (\ref{eq:definhxt}) satisfies the following partial 
differential equation:
\begin{equation}
 {\partial\over\partial t}\;h(x,t)=-\xi h(x,t)-{\partial\over\partial x}\;
 \Bigl\{-\alpha(x-\beta)\,
 h(x,t)\Bigr\}
 +{1\over 2}\,\sigma^2\,{\partial^2\over\partial x^2}h(x,t),
 \label{eq:equdiffsomma}
\end{equation}
with conditions
\begin{equation}
  \left.-\xi+\alpha \beta h(0,t)-\frac{\sigma^2}{2}\frac{\partial}{\partial x}h(x,t)\right|_{x=0}=0,
  \label{eq:equdiffbound1}
\end{equation}
\begin{equation}
 \lim_{x\rightarrow +\infty} h(x,t)=0,
  \label{eq:equdiffbound2}
\end{equation}
and Dirac-delta initial condition
\begin{equation*}
 \lim_{t\to 0^+}h(x,t)=\delta(x).
  \label{eq:initcond}
\end{equation*}
\end{proposition}
\begin{proof}
The proof of Eqs.\ (\ref{eq:equdiffsomma}) $\div$ (\ref{eq:equdiffbound2})
follows immediately from Proposition \ref{prop:diffj}, and recalling assumption (\ref{eq:definhxt}).
\end{proof}
Note that Eq.\ $(\ref{eq:equdiffsomma})$ is the Fokker-Planck equation for a  
jump Ornstein-Uhlenbeck diffusion process on $\mathbb{R}^+$ with drift $A_1(x)=-\alpha(x-\beta)$, infinitesimal variance $A_2(x)=\sigma^2$ 
and with a reflecting boundary in $x=0$ (see Eq. (\ref{eq:equdiffbound1})). We recall \cite{Giorno2012} for the study of a reflected OU process subject to catastrophes that occur at exponential rate 
reducing the state of the system instantaneously to zero.
\par
Let us denote by $\tilde{\cal Z}(t)$ an Ornstein-Uhlenbeck diffusion process on $\mathbb{R}^+$ with drift $A_1(x)=-\alpha(x-\beta)$, infinitesimal variance $A_2(x)=\sigma^2$ 
and with a reflecting boundary in $x=0$ and let us set 
$$
{\tilde r}(x,t|0):=\frac{{\rm d} }{{\rm d} x} {\mathbb P}(\tilde{\cal Z}(t)\leq x),
$$
its probability density function. 
It is well know (see, for instance, \cite{Giorno2012}) that, for $t \geq 0$, $x \geq 0$,  the processes $\tilde{\cal Z}(t)$ and ${\cal Z}(t)$ are linked 
by the following relations:
\begin{eqnarray}
  &&h(x,t)=e^{-\xi t}{\tilde r}(x,t|0)+\xi \int_0^t e^{-\xi \tau}{\tilde r}(x,\tau|0) {\rm d}\tau, 
  \label{rel_pdf_1} \\
  &&E[{\cal Z}^r(t)]=e^{-\xi t}E[{\tilde {\cal Z}}^r(t)]+\xi \int_0^t e^{-\xi \tau}E[{\tilde {\cal Z}}^r(\tau)] {\rm d}\tau, \nonumber
\end{eqnarray}
where $E[{\tilde {\cal Z}}^r(t)]$ ($E[{ {\cal Z}}^r(t)]$) is the $r-$th moment of ${\tilde {\cal Z}}(t)$ (${ {\cal Z}}(t)$).
\par
Making use of Eq. (\ref{rel_pdf_1}), in the following proposition we obtain the explicit expression of $h(x,t)$ in the special case $\lambda=\mu$.

\begin{proposition}
In the case $\beta=0$, for $t>0$ and $x\in {\mathbb R}^{+}$, the probability density function of ${\cal Z}(t)$ is given by
\begin{eqnarray}
&& \hspace*{-1.5cm}
h(x,t)=2{\rm e}^{-\xi t} \sqrt{\frac{\alpha}{\pi \sigma^2(1-{\rm e}^{-2 \alpha t})}} \exp\left\{-\frac{\alpha x^2}{\sigma^2(1-{\rm e}^{-2\alpha t})}\right\}
\nonumber
\\
&& \hspace*{-0.5cm}
+ \frac{x \xi}{\sqrt{\pi} \sigma^2} \sum_{n=0}^{+\infty} \frac{\left(1-\frac{\xi}{2 \alpha}\right)_n}{n!} \left(\frac{x^2 \alpha}{\sigma^2}  \right)^n
\Gamma\left[-\frac{1}{2}-n, \frac{x^2 \alpha (1+ \coth(t \alpha)}{2 \sigma^2} \right],\nonumber
\end{eqnarray}
where $\Gamma[a,z]$ denotes the upper incomplete gamma function.
\end{proposition}
\begin{proof}
Under the assumptions $\beta=0$, the expression of ${\tilde r}(x,t|0)$ is well known (see, for instance, \cite{Giorno2023}) as
\begin{equation*}
{\tilde r}(x,t|0)=2\sqrt{\frac{\alpha}{\pi \sigma^2(1-e^{-2 \alpha t})}}\exp\left\{-\frac{\alpha x^2}{\sigma^2(1-e^{-2\alpha t})}\right\}.
\end{equation*} 
Hence, the proof follows from Eq. (\ref{rel_pdf_1}) and noting that 
$$
\sum_{n=0}^{+\infty} \frac{\left(1-\frac{\xi}{2 \alpha}\right)_n}{n!} z^{-n}=\left(\frac{z-1}{z}\right)^{\frac{\xi}{2 \alpha}-1}.
$$
\end{proof}
\begin{remark}
In the special case $\beta=0$, if $\frac{\xi}{2 \alpha}=m$, $m\in {\mathbb N}$, it results 
\begin{eqnarray*}
&& \hspace*{-1.7cm}
h(x,t)=\exp\left\{-2 m t \alpha-\frac{\alpha x^2 [1+\coth (t \alpha)]}{2\sigma^2}\right\}
\frac{\sqrt{2 \alpha [1+\coth (t \alpha)]}}{\sqrt{\pi} \sigma}
\nonumber
\\
&& \hspace*{-0.5cm}
+ \frac{2 m x \alpha}{\sqrt{\pi} \sigma^2} \sum_{n=0}^{m-1} {m-1 \choose n}
 \left(-\frac{x^2 \alpha}{\sigma^2}  \right)^n
\Gamma\left[-\frac{1}{2}-n, \frac{x^2 \alpha (1+ \coth(t \alpha)}{2 \sigma^2} \right].
\end{eqnarray*}
\end{remark}

\begin{figure}[t] 
\begin{center}
\hspace*{-0.4cm}
\includegraphics[height=4.5cm,width=6.5cm]{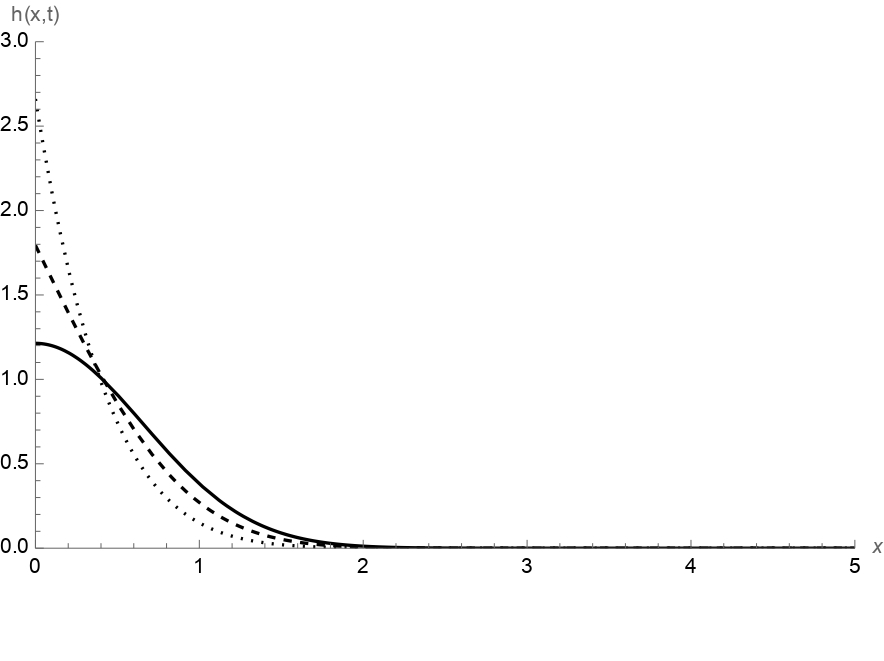}
\includegraphics[height=4.5cm,width=6.5cm]{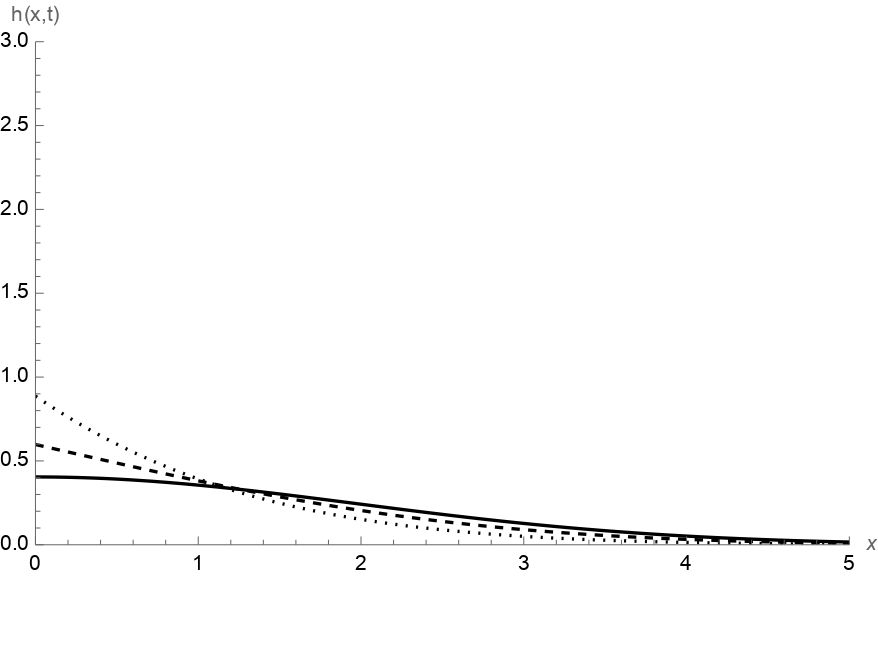}
\end{center}
\caption{Plots of $h(x,t)$ for $\beta=0$, $\alpha=1$, $t=1$, $\xi=0$ (solid line), $\xi=1$ (dashed line) and $\xi=3$ (dotted line), with $\sigma=1$ (left-hand side)
and $\sigma=3$ (right-hand side).}
\label{hbeta0}
\end{figure}

Figure \ref{hbeta0} presents some plots of $h(x,t)$ in the case $\lambda=\mu$ for different values of $\xi$ and $\sigma$. 
We observe that, as $\xi$ increases, the curve approaches the vertical axis and becomes more sharply peaked near zero. For fixed 
$\xi$, the curve broadens as $\sigma$ increases.

In the sequel, aiming to disclose the asymptotic probability density function of ${\cal Z}(t)$, we introduce the Laplace transform of the pdf 
of $\tilde {\cal Z}(t)$, defined for $p>0$ as 
 \begin{equation}
{\tilde r}_p (x|0)=\int_0^{+\infty}e^{-p t} {\tilde r}(x,t|0)dt.
    \label{dens_OU_refl_noC}
\end{equation}
The expression of ${\tilde r}_p(x|0)$ is well known in the literature \cite{Giorno2023}. 
However, in the following proposition, we provide an alternative expression of such Laplace transform which will be useful in the study of ${\cal Z}(t)$.
\begin{proposition}
For $p>0$ and $x>0$, the Laplace transform ($\ref{dens_OU_refl_noC}$) is given by
\begin{equation}
    {\tilde r}_p(x|0)=\exp\left\{-\frac{\alpha}{ \sigma^2}[(x-\beta)^2-\beta^2]\right\}\frac{\sqrt{\alpha}}{ \sigma p}\,\frac{H\left(-\frac{p}{\alpha},\frac{\sqrt{\alpha}}{\sigma}(x-\beta)\right)}{H\left(-1-\frac{p}{\alpha},-\frac{\sqrt{\alpha}}{\sigma}\beta\right)}, 
    \label{lapl_tr_r_tilde}
\end{equation}
where 
\begin{equation}
    H(\nu, z)= 2^\nu \sqrt{\pi} \left( 
\frac{1}{\Gamma\left(\frac{1 - \nu}{2}\right)} 
\, {}_1F_1\left(-\frac{\nu}{2}, \frac{1}{2}, z^2\right) 
- \frac{2z}{\Gamma\left(-\frac{\nu}{2}\right)} 
\, {}_1F_1\left(\frac{1 - \nu}{2}, \frac{3}{2}, z^2\right) 
\right)
\label{HermiteH}
\end{equation}
is the Hermite function. 
\end{proposition}
\begin{proof}
By making use of formula (4.13) of \cite{Giorno2023}, we obtain
\begin{eqnarray*}
&& \hspace*{-1.2cm}
{\tilde r}_p(x|0)=\frac{2^{\frac{p}{\alpha}-1}}{\pi \sigma \sqrt{\alpha}}\Gamma\left(\frac{p}{2 \alpha}\right)\Gamma\left(\frac{1}{2}+\frac{p}{2 \alpha}\right)
\exp\left\{-\frac{\alpha}{2 \sigma^2}[(x-\beta)^2-\beta^2]\right\} \frac{D_{-\frac{p}{\alpha}}\left(\frac{\sqrt{2\alpha}}{\sigma}
(x-\beta)\right)}{D_{-\frac{p}{\alpha}-1}\left(-\frac{\sqrt{2\alpha}}{\sigma}\beta\right)}
\nonumber\\
&& \hspace*{-0.8cm}
\times\left[D_{-\frac{p}{\alpha}-1}\left(-\frac{\sqrt{2\alpha}}{\sigma}\beta\right)D_{-\frac{p}{\alpha}}\left(\frac{\sqrt{2\alpha}}{\sigma}\beta\right)+D_{-\frac{p}{\alpha}-1}\left(\frac{\sqrt{2\alpha}}{\sigma}\beta\right)D_{-\frac{p}{\alpha}}\left(-\frac{\sqrt{2\alpha}}{\sigma}\beta\right)
\right],
\end{eqnarray*}
where $D_{\nu}(z)$ is the parabolic cylinder function. Hence, due to Eq. (53) of \cite{Nasri2016}, we obtain
\begin{equation*}
{\tilde r}_p(x|0)=\frac{2^{\frac{p}{\alpha}-1} \sqrt{2\pi}}{\pi \sigma \sqrt{\alpha}}
\frac{\Gamma\left(\frac{p}{2 \alpha}\right)\Gamma\left(\frac{1}{2}+\frac{p}{2 \alpha}\right)}{\Gamma\left(1+\frac{p}{\alpha}\right)}
\exp\left\{-\frac{\alpha}{2 \sigma^2}[(x-\beta)^2-\beta^2]\right\} \frac{D_{-\frac{p}{\alpha}}\left(\frac{\sqrt{2\alpha}}{\sigma}
(x-\beta)\right)}{D_{-\frac{p}{\alpha}-1}\left(-\frac{\sqrt{2\alpha}}{\sigma}\beta\right)},
\label{lapl_transf_r_tilde}
\end{equation*}
so that the result follows after some straightforward calculations.
\end{proof}
%
Let us denote by $w(x)$ the asymptotic distribution of ${\cal Z}(t)$, i.e.:
\begin{equation}
w(x):=\frac{{\rm d}}{{\rm d}x}{\mathbb P}({\cal Z}\leq x)= \lim_{t \rightarrow +\infty} r(x,t|0),
\label{def_w}
\end{equation}
where ${\cal Z}$ is the random variable describing the stationary state.
The explicit expression of $w(x)$ is obtained in the following proposition. 
\begin{proposition}
The asymptotic distribution of ${\cal Z}(t)$ is given by
\begin{equation}
w(x)=\frac{\sqrt{\alpha}\exp\left\{-\frac{\alpha}{ \sigma^2}[(x-\beta)^2-\beta^2]\right\} H\left(-\frac{\xi}{\alpha},\frac{\sqrt{\alpha}(x-\beta)}{\sigma}\right)}{\sigma\, H\left(-1-\frac{\xi}{\alpha},-\frac{\sqrt{\alpha}\beta}{\sigma}\right)}, \quad x>0,
\label{dens_asint}
\end{equation}
where the Hermite function $H(\nu,z)$ has been defined in Eq. (\ref{HermiteH}).
\end{proposition}
\begin{proof}
From Eq. (\ref{def_w}) and recalling Eq. (\ref{rel_pdf_1}), we have that 
$$
w(x)=\xi \int_0^{+\infty} e^{-\xi \tau}{\tilde r}(x,\tau|0)d\tau=\xi\,{\tilde r}_\xi(x|0),
$$
where ${\tilde r}_\xi(x|0)$ has been provided  in Eq. (\ref{lapl_tr_r_tilde}). Hence the proof follows by means of straightforward calculations.
\end{proof}
\begin{remark}
As a confirmation of the results obtained, we demonstrate that the integral of $w(x)$ over $[0,+\infty]$ is unitary. 
Indeed, we have
\begin{eqnarray*}
&& \hspace*{-1cm}
\int_0^{+\infty}w(x)dx
=\frac{e^{\frac{\alpha \beta^2}{\sigma^2}}}{H\left(-1-\frac{\xi}{\alpha},-\frac{\sqrt{\alpha}\beta}{\sigma}\right)}\int_{-\frac{\beta \sqrt{\alpha}}{\sigma}}^{+\infty}e^{-z^2}H\left(-\frac{\xi}{\alpha},z\right){\rm d}z
\nonumber\\
&& \hspace*{-1cm}
=\frac{1}{2 H\left(-1-\frac{\xi}{\alpha},-\frac{\sqrt{\alpha}\beta}{\sigma}\right) \sigma\Gamma\left(1+\frac{\xi}{\alpha}\right)}\left[2\beta \sqrt{\alpha}\Gamma\left(\frac{\xi}{2\alpha}+1\right)e^{\frac{\alpha \beta^2}{\sigma^2}}
{}_{1}F_{1}\left(\frac{1}{2}-\frac{\xi}{2\alpha},\frac{3}{2},-\frac{\alpha \beta^2}{\sigma^2}\right)\right.
\nonumber\\
&& \hspace*{-1cm}
\left.+\sigma\Gamma\left(\frac{\xi}{2\alpha}+\frac{1}{2}\right)e^{\frac{\alpha \beta^2}{\sigma^2}}
{}_{1}F_{1}\left(-\frac{\xi}{2\alpha},\frac{1}{2},-\frac{\alpha \beta^2}{\sigma^2}\right)
\right]\nonumber\\
&& \hspace*{-1cm}
=\frac{1}{2\Gamma\left(\frac{\xi}{\alpha}+1\right)}\frac{1}{H\left(-1-\frac{\xi}{\alpha},-\frac{\sqrt{\alpha}\beta}{\sigma}\right)}\left[\frac{2 \beta \sqrt{\alpha}}{\sigma}\Gamma\left(\frac{\xi}{2\alpha}+1\right){}_{1}F_{1}\left(1+\frac{\xi}{2\alpha},\frac{3}{2},\frac{\alpha \beta^2}{\sigma^2}\right)\right.
\nonumber\\
&& \hspace*{-1cm}
\left.+\Gamma\left(\frac{\xi}{2\alpha}+\frac{1}{2}\right){}_{1}F_{1}\left(\frac{1}{2}+\frac{\xi}{2\alpha},\frac{1}{2},\frac{\alpha \beta^2}{\sigma^2}\right)
\right]=1,
\end{eqnarray*}
due to the definition of the Hermite function.
\end{remark}

\begin{figure}[t] 
\begin{center}
\hspace*{-0.4cm}
\includegraphics[height=4.5cm,width=7.5cm]{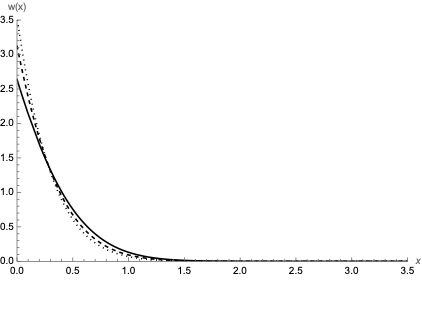}
\includegraphics[height=4.5cm,width=7.5cm]{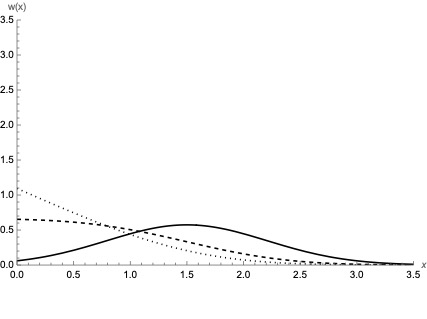}
\end{center}
\caption{Plots of $w(x)$ for $\sigma=1$, $\xi=0$ (solid line), $\xi=1$ (dashed line) and $\xi=2$ (dotted line), with $\alpha=1$, $\beta=-1$ (left-hand side)
and $\alpha=1$, $\beta=1.5$ (right-hand side).}
\label{plotw}
\end{figure}

Figure \ref{plotw} shows the asymptotic cumulative probability density function for different values of the involved parameters. 
As in the case of transient regime, the curve approaches the vertical axis becoming more peaked near the origin when $\xi$ increases. 
We note that for fixed $\xi$, as $\beta$ increases, the entire distribution translates to the right. 

\begin{proposition}
 The first and second asymptotic moments of ${\cal Z}(t)$ are given by
\begin{eqnarray*}
&& \hspace*{-0.8cm}
E[{\cal Z}]=
\frac{2^{-\xi/\alpha} \sqrt{\pi} \beta}{H\left(-1-\frac{\xi}{\alpha},-\frac{\sqrt{\alpha}\beta}{\sigma}\right)}\left[
\frac{\sigma}{4 \beta \sqrt{\alpha} \Gamma\!\left(\frac{3}{2}+\frac{\xi}{2 \alpha} \right)}
{}_{1}F_{1}\left(1+\frac{\xi}{2\alpha},\frac{1}{2},\frac{\alpha \beta^2}{\sigma^2}\right)\right.
\nonumber
\\
&& \hspace*{-0.8cm}
\left.
+\frac{1}{2 \Gamma\!\left(1+\frac{\xi}{2 \alpha} \right)} {}_{1}F_{1}\left(\frac{1}{2}+\frac{\xi}{2\alpha},\frac{1}{2},\frac{\alpha \beta^2}{\sigma^2}\right)
-\frac{\beta^2 \xi}{3 \sigma^2 \Gamma\!\left(1+\frac{\xi}{2 \alpha}\right)}
{}_{1}F_{1}\left(\frac{3}{2}+\frac{\xi}{2\alpha},\frac{5}{2},\frac{\alpha \beta^2}{\sigma^2}\right)\right],
\end{eqnarray*}
\begin{eqnarray*}
&& \hspace*{-0.8cm}
E[{\cal Z}^2]=
\frac{\sigma}{4 \beta \alpha^{7/2}\Gamma\!\left(3+\frac{\xi}{\alpha}\right)H\left(-1-\frac{\xi}{\alpha},-\frac{\sqrt{\alpha}\beta}{\sigma}\right)}
\left[ -\xi (\alpha+\xi) \sigma^2 \Gamma\!\left(\frac{\xi}{2 \alpha} \right)  {}_{1}F_{1}\left(\frac{\xi}{2\alpha},\frac{1}{2},\frac{\alpha \beta^2}{\sigma^2}\right)
\right.
\nonumber\\
&& \hspace*{-0.8cm}
\left. +\xi \, [(\alpha+\xi) \sigma^2 +2 \alpha^2 \beta^2] \, \Gamma\!\left(\frac{\xi}{2 \alpha} \right) {}_{1}F_{1}\left(1+\frac{\xi}{2\alpha},\frac{1}{2},\frac{\alpha \beta^2}{\sigma^2}\right)
+4 \alpha^{5/2} \beta \sigma \Gamma\!\left(\frac{3}{2}+\frac{\xi}{2\alpha}\right)  {}_{1}F_{1}\left(\frac{3}{2}+\frac{\xi}{2\alpha},\frac{1}{2},\frac{\alpha \beta^2}{\sigma^2}\right)\right],
\end{eqnarray*}
where the Hermite function is given in Eq. (\ref{HermiteH}).
\label{propAsyMom}
\end{proposition}
\begin{proof}
The results follow from Eq. (\ref{dens_asint}) noting that, for $n\geq 1$,
\begin{equation*}
\int_{0}^{+\infty} x^n w(x) {\rm d}x=
\frac{\sqrt{\alpha}}{\sigma H\left(-1-\frac{\xi}{\alpha},-\frac{\sqrt{\alpha}\beta}{\sigma}\right)} {\rm e}^{\frac{\alpha \beta^2}{\sigma^2}}
\sum_{j=0}^n {n \choose j}\beta^{n-j}\left(\frac{\sigma}{\sqrt{\alpha}}\right)^{j+1}
\int_{-\frac{\beta \sqrt{\alpha}}{\sigma}}^{+\infty}   y^j {\rm e}^{-y^2}
H\left(-\frac{\xi}{\alpha},y\right) {\rm d}y.
\end{equation*}
\end{proof}

\begin{remark}
Note that in the case of absence of catastrophes, the expected value and the second order moment are given by
\begin{eqnarray*}
&& \hspace*{-0.5cm}
E[{\cal Z}]\big\vert_{\xi=0}=\beta+ \frac{\sigma {\rm e}^{-\alpha \frac{\beta^2}{\sigma^2}}}{\sqrt{\pi \alpha} \left(1+Erf(\sqrt{\alpha}\beta/\sigma) \right)},
\\
&& \hspace*{-0.5cm}
E[{\cal Z}^2]]\big\vert_{\xi=0}=\beta^2+\frac{\sigma^2}{2 \alpha}+ \frac{\sigma \beta {\rm e}^{-\alpha \frac{\beta^2}{\sigma^2}}}{\sqrt{\pi \alpha} 
\left(1+Erf(\sqrt{\alpha}\beta/\sigma) \right)},
\end{eqnarray*}
which identify with Eq. (4.15) of \cite{Giorno2023}.
\par
Moreover, in the case $\beta=0$, the expected value and the second order moment take a very simple expression given by
the ratio of gamma functions
\begin{eqnarray*}
&& \hspace*{-0.5cm}
E[{\cal Z}]\big\vert_{\beta=0}=\frac{\sigma \Gamma\left[1+\frac{\xi}{2 \alpha}\right]}
{2 \sqrt{\alpha} \Gamma\left[\frac{3}{2}+\frac{\xi}{2 \alpha}\right]},
\\
&& \hspace*{-0.5cm}
E[{\cal Z}^2]\big\vert_{\beta=0}=\frac{2^{1+\frac{\xi}{\alpha}} \sigma^2 \Gamma\left[\frac{3}{2}+\frac{\xi}{2 \alpha}\right] \Gamma\left[1+\frac{\xi}{2 \alpha}\right]}
{\sqrt{\pi} \alpha \Gamma\left[3+\frac{\xi}{\alpha}\right]}.
\end{eqnarray*}
\end{remark}

\begin{figure}[t]
\centering
\includegraphics[height=6cm,width=14cm]{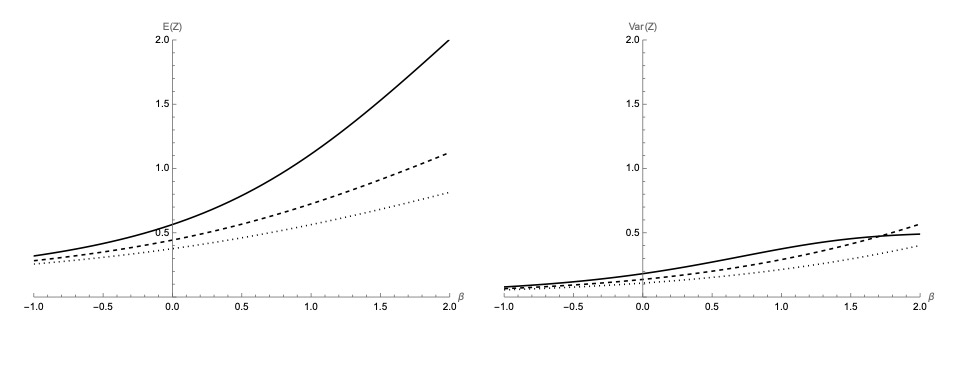}
\caption{The asymptotic mean and variance for $\sigma=1$ and $\alpha=1$, with $\xi=0$ (solid line), $\xi=1$ (dashed line) and $\xi=2$ (dotted line).}
\label{momentiw}
\end{figure}

Figure \ref{momentiw} presents some plots illustrating the behavior of the asymptotic mean and variance as functions of $\beta$, derived from Proposition \ref{propAsyMom} by considering
different values of $\xi$. Both the asymptotic mean and variance are increasing functions of $\beta$. Furthermore, the expected value takes larger values as $\xi$ decreases.
Regarding the asymptotic variance, we observe a distinct behavior based on the value of $\xi$: when $\xi=0$, the asymptotic variance tends to a constant value given by $\frac{1}{2 \alpha}$ as $\beta \to +\infty$; 
in the case $\xi>0$, the asymptotic variance diverges as $\beta \to +\infty$.

Let us denote by $w(x,j)$ the asymptotic (sub)density related to the $j$-th ray of the spider, i.e.
\begin{equation}
w(x,j):=\lim_{t \rightarrow +\infty} f(x,j,t),\qquad x \in \mathbb{R}^+\cup\{0\},\; j \in D.
\label{asymp_dens_j}
\end{equation}
Note that, due to Eq. (\ref{def_w}), it is 
\begin{equation}
 w(x)=\sum_{j\in D} w(x,j), \qquad x \in \mathbb{R}^+\cup\{0\}.
\label{defasymp_dens_j}
\end{equation}
\begin{proposition}
For all $\alpha,\sigma >0$, $\beta \in \mathbb R$ and $x \in \mathbb{R}^+\cup\{0\}$, $j \in D$, the asymptotic density (\ref{asymp_dens_j}) 
is given by
$$
w(x,j)=\pi_j \,w(x), 
$$
where $w(x)$ is given in (\ref{dens_asint}) and $\vec \pi=(\pi_1, \ldots, \pi_d)$ is the invariant distribution of the stochastic matrix $C$ and the 
Hermite function has been defined in Eq. (\ref{HermiteH}).
\end{proposition}
\begin{proof}
As $t \rightarrow +\infty$, Eq.\ (\ref{eq:equdiff}) can be written as  
\begin{equation}
0=-{\partial\over\partial x}\;
 \Bigl\{-\alpha(x-\beta)\,
 w(x,j)\Bigr\}
 +{1\over 2}\,\sigma^2\,{\partial^2\over\partial x^2}w(x,j) -\xi w(x,j), 
 \label{eq:equdiff_wj}
\end{equation}
whereas the conditions (\ref{eq:rifless}), (\ref{cond_approx_diff1}) and (\ref{cond_approx_diff2}) become
$$
\alpha \beta \sum_{j=1}^d w(0,j)-{\sigma^2\over 2} \frac{\partial}{\partial x} \sum_{j=1}^d w(x,j)\big\vert_{x=0}-\xi=0,
$$
\begin{equation}
 w(0,j)=\sum_{l=1}^d c_{l,j}\,w(0,l),\qquad j\in D,
 \label{eq:woj}
\end{equation}
$$
\lim_{x\rightarrow +\infty} w(x,j)=0.
$$
From Eq. (\ref{eq:woj}), recalling that $w(0)=\sum_{l=1}^d w(0,l)$, we have that 
$$
\pi_j:=\frac{w(0,j)}{w(0)}
$$
is the invariant distribution for the stochastic matrix $C$. 
\par 
Hence, from Eq. (\ref{eq:equdiff_wj}), being $w(0,j)=\pi_j w(0)$, and recalling Eq. (\ref{defasymp_dens_j}), 
the proof finally follows. 
\end{proof}
\begin{example}\rm
Recalling (\ref{c_lj}) and (\ref{eq:woj}), we can consider some examples of the matrix $C$,  
which regulates the switching mechanism for the particle phases, 
and the corresponding  vector $\vec \pi=(\pi_1, \ldots, \pi_d)$ of the invariant distribution.
\begin{enumerate}
\item 
The transitions from line $l$ to line $j$ occur uniformly:
$$
 c_{l,j}=\frac{1}{d},\qquad \forall \, l,j \in D.
$$

\item 
The transitions occur cyclically clockwise:
$$
c_{l,j}=\left\{
 \begin{array}{ll}
 1, & \quad j=l+1, \\[1mm]
 0, & \quad \hbox{otherwise} 
\end{array}
\right.
\quad (l=1,2,\ldots,d-1), 
\qquad 
c_{d,j}=\left\{
 \begin{array}{ll}
 1, & \quad j=1,  \\[1mm]
 0, & \quad \hbox{otherwise}. 
\end{array}
\right.
$$
In both these cases the invariant distribution is uniform
$\vec \pi=\left(\frac{1}{d},\,\frac{1}{d}, \ldots,\frac{1}{d}\right)$.
\item  
The transitions occur on adjacent lines, according to a random-walk scheme: 
%
$$
c_{1,j}=\left\{
 \begin{array}{ll}
 1, & \quad j=2, \\[1mm]
 0, & \quad \hbox{otherwise} 
\end{array}
\right.
\quad  
c_{l,j}=\left\{
 \begin{array}{ll}
 1-p, & \quad j=l-1,\\[1mm]
p, & \quad j=l+1, \\[1mm]
 0, & \quad \hbox{otherwise} 
\end{array}
\right.
\quad 
(l=2,3,\ldots,d-1), 
$$
$$
c_{d,j}=\left\{
 \begin{array}{ll}
 1, & \quad j=1,\\[1mm]
 0, & \quad \hbox{otherwise}. 
\end{array}
\right.
$$
The invariant distribution $\vec \pi$ has components
$$
 \pi_1=1-\frac{1}{2(1+p^2)} \frac{1-p^{d-1}}{1-p},
  \qquad 
 \pi_j=\frac{p^{j-2}}{2(1+p^2)},\quad j=2,3,\ldots,d.
 $$
\end{enumerate}
\end{example}

\section{Conclusions}
In this paper we have introduced and analyzed a finite QBD process subject to catastrophic events, extending the framework previously considered in \cite{DiCrescenzo2022}.  In particular, we consider level-dependent
one-step transition probabilities and assume that the transitions from a certain state are allowed only to the two adjacent levels within the same phase. However, from the level $0$ of any phase
we can have one-step transitions to the level $1$ of the same phase or of any other phase, which are regulated by an irreducible stochastic matrix. Moreover, 
catastrophes occur according to a Poisson process and instantaneously reset the system to the level $0$ of the current phase, producing a dynamics that combines gradual, level-structured evolution with sudden discontinuities.  
This mechanism provides a flexible stochastic representation for systems in which regular transitions are occasionally interrupted by abrupt resets, a phenomenon relevant in population dynamics, queueing theory, and networked systems.

By employing generating function techniques and relations between the process with catastrophes and the process without catastrophes, we characterized both the transient and long-term behavior of the process.  
Our analysis yields closed-form expressions for several relevant descriptors, and highlights how the catastrophic component influences the behavior of the process.   
Furthermore, by considering a diffusive approximation, we provided insight into the large-scale behavior of the system, offering an interpretable continuous limit that captures the interplay between regular dynamics and sudden collapses.

Finally, the results presented in this work contribute to the growing literature on stochastic models with catastrophic events, illustrating how the rich structure of QBD processes can be effectively combined with sudden-reset dynamics to capture realistic phenomena characterized by both persistence and abrupt change.


\appendix	
\section{Proof of Proposition \ref{asymptotic}}
\label{proof_prop_asym}

Let us evaluate the Laplace transform of $p(0,t)$. By requiring that $lim_{z \rightarrow 0^+}F(z,t)=0$, from Eq. (\ref{F(z,t)}) we obtain 
\begin{eqnarray}
  &&  \frac{e^{-\xi t}(\lambda+\mu e^{t(\lambda+\mu)})^N(-\mu+\mu e^{t(\lambda+\mu)})^N}{e^{2 N t(\lambda+\mu)}(\lambda+\mu)^{2N}}\nonumber\\
  &&+\frac{\xi}{\lambda+\mu}\int_0^t e^{-\xi(t-y)}(\mu-\mu e^{-(t-y)(\lambda+\mu)})^N(\mu+\lambda e^{-(t-y)(\lambda+\mu)})^N dy\nonumber\\
  &&-\frac{N \mu}{(\lambda+\mu)^{2N-1}}\int_0^t p(0,y)e^{-(t-y)(\lambda+\mu+\xi)}(\mu-\mu e^{-(t-y)(\lambda+\mu)})^{N-1}(\mu+\lambda e^{-(t-y)(\lambda+\mu)})^N dy=0,
 \nonumber
\end{eqnarray}
so that
\begin{eqnarray}
&&N(\lambda+\mu)\int_0^t p(0,y)e^{-(t-y)(\lambda+\mu+\xi)}(1- e^{-(t-y)(\lambda+\mu)})^{N-1}(\mu+\lambda e^{-(t-y)(\lambda+\mu)})^N dy\nonumber\\
&&=e^{-\xi t}(\lambda e^{-t(\lambda+\mu)}+\mu )^N(1- e^{-t(\lambda+\mu)})^N+\xi \int_0^t e^{-\xi(t-y)}(1- e^{-(t-y)(\lambda+\mu)})^N(\mu+\lambda e^{-(t-y)(\lambda+\mu)})^N dy.\nonumber 
\end{eqnarray}
Denoting by ${\cal L}_\eta (\cdot)$, $\eta>0$, the Laplace transform operator, and applying it on both sides of the previous equation, we obtain
\begin{eqnarray}
&& \hspace*{-0.1cm}
N(\lambda+\mu){\cal L}_{\lambda+\mu+\xi+\eta}\left[(1- e^{-t(\lambda+\mu)})^{N-1}(\mu+\lambda e^{-t(\lambda+\mu)})^N\right]{\cal L}_{\eta}[p(0,t)]
\nonumber\\
&& \hspace*{-0.1cm}
={\cal L}_{\eta+\xi}\left[(1- e^{-t(\lambda+\mu)})^N(\mu+\lambda e^{-t(\lambda+\mu)})^N\right]
+\frac{\xi}{\eta}{\cal L}_{\eta+\xi}\left[(1- e^{-t(\lambda+\mu)})^{N}(\mu+\lambda e^{-t(\lambda+\mu)})^N\right],\nonumber
\end{eqnarray}
so that 
\begin{equation*}
{\cal L}_{\eta}[p(0,t)]=\frac{{\cal L}_{\eta+\xi}\left[(1- e^{-t(\lambda+\mu)})^N(\mu+\lambda e^{-t(\lambda+\mu)})^N\right]\left(1+\frac{\xi}{\eta}\right)}{N(\lambda+\mu){\cal L}_{\lambda+\mu+\xi+\eta}\left[(1- e^{-t(\lambda+\mu)})^{N-1}(\mu+\lambda e^{-t(\lambda+\mu)})^N\right]}. 
\end{equation*}
Hence, due to Equation (28) of \cite{Prudnikov}, after some computations, we get the following expression of the Laplace transform of $p(0,t)$ 
\begin{equation}
{\cal L}_{\eta}[p(0,t)]=\frac{\left(1+\frac{\xi}{\eta}\right)
{}_{2}F_{1}\left(\frac{\eta+\xi}{\lambda+\mu},-N,\frac{\eta+\xi}{\lambda+\mu}+N+1;-\frac{\lambda}{\mu}\right)}{(\eta+\xi){}_{2}F_{1}\left(\frac{\eta+\xi}{\lambda+\mu}+1,-N,\frac{\eta+\xi}{\lambda+\mu}+N+1;-\frac{\lambda}{\mu}\right)}. 
\label{lapl_trans_p0}
\end{equation}
By applying the Laplace transform operator to Eq. (\ref{F(z,t)}), we obtain 
\begin{eqnarray}
&& \hspace*{-0.8cm}
{\cal L}_\eta[F(z,t)]=\int_0^\infty e^{-\eta t}F(z,t)dt=\frac{\lambda^N(1-z)^N}{z^N(\lambda+\mu)^{2N}}\sum_{j=0}^N {N \choose j}[(\lambda+\mu)z]^{N-j}\mu^j(1-z)^j
\nonumber\\
&& \hspace*{-0.8cm}
\times{\cal L}_{\eta+\xi}\left[(1- e^{-t(\lambda+\mu)})^j\left(\frac{\lambda z+\mu}{\lambda(1-z)}+e^{-t(\lambda+\mu)}\right)^N\right]  
-\frac{N\mu(1-z)^{N+1}\lambda^N}{(\lambda+\mu)^{2N-1}z^N}\sum_{j=0}^{N-1}{N-1 \choose j}
\nonumber\\
&& \hspace*{-0.8cm}
\times
[(\lambda+\mu)z]^{N-j-1}\mu^j(1-z)^j
{\cal L}_{\eta}[p(0,t)]{\cal L}_{\eta+\xi+\lambda+\mu}\left[(1- e^{-t(\lambda+\mu)})^j\left(\frac{\lambda z+\mu}{\lambda(1-z)}+e^{-t(\lambda+\mu)}\right)^N\right]
\nonumber\\
&& \hspace*{-0.8cm}
+\frac{\xi \lambda^N (1-z)^N}{(\lambda+\mu)^{2N}z^N}\sum_{j=0}^N {N \choose j}(\lambda+\mu)^{N-j}z^{N-j}\mu^j(1-z)^j
\frac{1}{\eta}{\cal L}_{\eta+\xi}\left[(1- e^{-t(\lambda+\mu)})^j\left(\frac{\lambda z+\mu}{\lambda(1-z)}+e^{-t(\lambda+\mu)}\right)^N\right].
\nonumber
\end{eqnarray}
Recalling Eq. (\ref{lapl_trans_p0}) and by making use of Equation (28) of \cite{Prudnikov}, after some calculations we get 
\begin{eqnarray}
{\cal L}_\eta[F(z,t)]&=&\frac{\lambda^N(1-z)^N}{z^N(\lambda+\mu)^{2N+1}}\sum_{j=0}^N {N \choose j}[(\lambda+\mu)z]^{N-j}\mu^j(1-z)^j\nonumber\\
&\times& Beta\left(j+1,\frac{\eta+\xi}{\lambda+\mu}\right)\left(\frac{\lambda z+\mu}{\lambda(1-z)}\right)^N{}_{2}F_{1}\left(\frac{\eta+\xi}{\lambda+\mu},-N,\frac{\eta+\xi}{\lambda+\mu}+j+1;-\frac{\lambda(1-z)}{\lambda z+\mu}\right)\nonumber\\
&-&\frac{N\mu(1-z)^{N+1}\lambda^N}{(\lambda+\mu)^{2N}z^N}\sum_{j=0}^{N-1}{N-1 \choose j}
[(\lambda+\mu)z]^{N-j-1}\mu^j(1-z)^j{\cal L}_{\eta}[p(0,t)]\nonumber\\
& \times&Beta\left(j+1,\frac{\eta+\xi}{\lambda+\mu}+1\right)\left(\frac{\lambda z+\mu}{\lambda(1-z)}\right)^N{}_{2}F_{1}\left(\frac{\eta+\xi}{\lambda+\mu}+1,-N,\frac{\eta+\xi}{\lambda+\mu}+j+2;-\frac{\lambda(1-z)}{\lambda z+\mu}\right)\nonumber\\
&+&\frac{\xi \lambda^N (1-z)^N}{(\lambda+\mu)^{2N+1}z^N}\sum_{j=0}^N {N \choose j}(\lambda+\mu)^{N-j}z^{N-j}\mu^j(1-z)^j\nonumber\\
&\times& \frac{1}{\eta} Beta\left(j+1,\frac{\eta+\xi}{\lambda+\mu}\right)\left(\frac{\lambda z+\mu}{\lambda(1-z)}\right)^N{}_{2}F_{1}\left(\frac{\eta+\xi}{\lambda+\mu},-N,\frac{\eta+\xi}{\lambda+\mu}+j+1;-\frac{\lambda(1-z)}{\lambda z+\mu}\right),\nonumber
\end{eqnarray}
where $Beta(z_1,z_2)$ denotes the Beta function. Finally, recalling that $F(z)=\lim_{\eta \rightarrow 0} \eta {\cal L}_\eta [F(z,t)]$ due the Tauberian theorem, the proof finally follows.

\section{Proof of Proposition \ref{prob_asympt_l_equal_m}}
\label{proof_prob_asympt_l_equal_m}
%
From Eq. (\ref{F_asympt}), by assuming $\lambda=\mu$, after some calculations we obtain
%
\begin{eqnarray}
&& \hspace*{-1cm}
F(z)= \sum_{h=0}^{N} z^{N - h} \left(\frac{1}{2}\right)^{N - h} \frac{N!}{h! \left( 1 + \frac{\xi}{2\lambda} \right)_{N-h}} \, {}_2F_1 \left( -h, 1 + N - h, 1 + N - h + \frac{\xi}{2 \lambda}, \frac{1}{2} \right)
+ \left( \frac{1}{4} \right)^N \sum_{h=0}^{2 N} z^{N - h} \binom{2N}{h} 
\nonumber 
\\
&& \hspace*{-1cm}
\times \sum_{k=0}^{N} {N \choose k} (-1)^k \left( 1 - \frac{k}{k + \frac{\xi}{4\lambda}} \right) {}_2F_1 \left( -2k, -h, -2N, 2 \right)
- \left( \frac{1}{2} \right)^N \sum_{h=0}^{N} z^{N - h} \binom{N}{h} \sum_{k=0}^{N} {N \choose k}  (-1)^k 
\nonumber\\
&& \hspace*{-1cm}
\times  \left( 1 - \frac{k}{k + \frac{\xi}{2\lambda}} \right) {}_2F_1 \left( -k, -h, -N, 2 \right)
- \frac{\Gamma \left( 1 + \frac{\xi}{4\lambda} \right) \Gamma \left( N + \frac{1}{4} \left( 2 + \frac{\xi}{\lambda} \right) \right)}{\Gamma \left( 1 + N + \frac{\xi}{4\lambda} \right) \Gamma \left( \frac{1}{4} \left( 2 + \frac{\xi}{\lambda} \right) \right) + \Gamma \left( 1 + \frac{\xi}{4\lambda} \right) \Gamma \left( N + \frac{1}{4} \left( 2 + \frac{\xi}{\lambda} \right) \right)}
\nonumber\\
&& \hspace*{-1cm}
\times \left( \frac{1}{4} \right)^N \left[\frac{1}{1 + \frac{\xi}{4\lambda}} \sum_{h=0}^{2N} z^{N-h} \binom{2N}{h} \sum_{k=0}^{N-1} \frac{N!}{k!(N-1-k)!} (-1)^k \left( 1 - \frac{k}{\frac{\xi}{4\lambda} + 1 + k} \right) {}_2F_1 \left( -2(1 + k), -h, -2N, 2 \right)\right.
\nonumber\\
&& \hspace*{-1cm}
-\left. \frac{2}{1 + \frac{\xi}{2\lambda}} \sum_{h=0}^{2N} z^{N-h} \binom{2N}{h} \sum_{k=0}^{N-1} \frac{N!}{k!(N-1-k)!} (-1)^k \left( 1 - \frac{k}{\frac{\xi}{4\lambda} + \frac{1}{2} + k} \right) {}_2F_1 \left( -1 - 2k, -h, -2N, 2 \right)\right].
\label{F_asymp_lambda_equal_mu}
\end{eqnarray}

By considering the coefficient of $z^0$ we have
\begin{eqnarray}
&& \hspace*{-0.8cm}
\rho(0)={}_2F_1\left( 1, -N, 1 + \frac{\xi}{2\lambda}, \frac{1}{2} \right)
+ \left( \frac{1}{4} \right)^N \binom{2N}{N} \sum_{k=0}^{N} {N \choose k} (-1)^k \left( 1 - \frac{k}{k + \frac{\xi}{4\lambda}} \right)
\frac{ 2^{2N} N! \, \left( \frac{1 - 2k}{2}\right)_N}{(2N)!}
\nonumber\\
&& \hspace*{-1cm}
- \left( \frac{1}{2} \right)^N \sum_{k=0}^{N} {N \choose k} \left( 1 - \frac{k}{k + \frac{\xi}{2\lambda}} \right)
- \frac{\Gamma \left( 1 + \frac{\xi}{4\lambda} \right) \Gamma \left( N + \frac{1}{4} \left( 2 + \frac{\xi}{\lambda} \right) \right)}{\Gamma \left( 1 + N + \frac{\xi}{4\lambda} \right) \Gamma \left( \frac{1}{4} \left( 2 + \frac{\xi}{\lambda} \right) \right) + \Gamma \left( 1 + \frac{\xi}{4\lambda} \right) \Gamma \left( N + \frac{1}{4} \left( 2 + \frac{\xi}{\lambda} \right) \right)}
\nonumber\\
&& \hspace*{-1cm}
\times 
\left( \frac{1}{4} \right)^N\left[ \frac{1}{1 + \frac{\xi}{4\lambda}} \binom{2N}{N} \sum_{k=0}^{N-1} \frac{N!}{k!(N-1-k)!} (-1)^k \left( 1 - \frac{k}{\frac{\xi}{4\lambda} + 1 + k} \right)
\frac{ 2^{2N} N! \, \left( \frac{1 - 2(1+k)}{2} \right)_N }{(2N)!}\right.
\nonumber\\
&& \hspace*{-1cm}
-\left. \frac{2}{1 + \frac{\xi}{2\lambda}} \binom{2N}{N} \sum_{k=0}^{N-1} \frac{N!}{k!(N-1-k)!} (-1)^k \left( 1 - \frac{k}{\frac{\xi}{4\lambda} + \frac{1}{2} + k} \right)
\frac{ 2^{2N} N! \, \left( -k \right)_N }{(2N)!}\right].
\nonumber
\label{rho_l_equal_mu_1}
\end{eqnarray}
The evaluation of the involved summations yields
$$
\rho(0)={}_2F_1\left( 1, -N, 1 + \frac{\xi}{2\lambda}, \frac{1}{2} \right) 
+ \frac{2}{1 + \frac{\left( 1 + \frac{\xi}{4\lambda} \right)_N}{\left( \frac{1}{4} \left( 2 + \frac{\xi}{\lambda} \right)\right)_N}}- 1 \nonumber\\
+ \frac{2\lambda N}{2\lambda + \xi}\cdot {}_2F_1\left( 1, 1 + \frac{\xi}{2\lambda} + N, 2 + \frac{\xi}{2\lambda}, -1 \right),
$$
so that Eq. (\ref{prob_asympt_0_l_equal_m}) follows after some computations.
\par
For $r=1,2,\ldots,N$, we consider the coefficient of $z^r$ in (\ref{F_asymp_lambda_equal_mu}), so that

\begin{eqnarray}
&& \hspace*{-1cm}
\rho(r)=\left(\frac{1}{2}\right)^r\frac{N!}{(N - r)! \,\left( 1 + \frac{\xi}{2\lambda} \right)_r} \, 
{}_2F_1\left( 1 + r, r - N, 1 + r + \frac{\xi}{2\lambda}, \frac{1}{2} \right) 
\label{rho_l_equal_mu_2}
\\
&& \hspace*{-1cm}
-\left( \frac{1}{2} \right)^N \binom{N}{N - r} 
\sum_{k=0}^{N} {N \choose k} (-1)^k \left( 1 - \frac{k}{k + \frac{\xi}{2\lambda}} \right) {}_2F_1\left( -k,r-N, -N, 2 \right)
\nonumber\\
&& \hspace*{-1cm}
+ \left( \frac{1}{4} \right)^N \binom{2N}{N - r} 
\sum_{k=0}^{N} {N \choose k} (-1)^k \left( 1 - \frac{k}{k + \frac{\xi}{4\lambda}} \right) {}_2F_1\left( -2k, r-N, -2N, 2 \right)
\nonumber\\
&& \hspace*{-1cm}
-\left(\frac{1}{4} \right)^N \frac{\Gamma\left( 1 + \frac{\xi}{4\lambda} \right) \Gamma\left( N+ \frac{1}{4} \left(2 + \frac{\xi}{\lambda} \right)\right)}{\Gamma\left( 1 + N + \frac{\xi}{4\lambda} \right) \Gamma\left( \frac{1}{4} \left( 2 + \frac{\xi}{\lambda} \right) \right) 
    + \Gamma\left( 1 + \frac{\xi}{4\lambda} \right) \Gamma\left( N+ \frac{1}{4} \left( 2 + \frac{\xi}{\lambda} \right) \right)}
    \nonumber\\
    && \hspace*{-1cm}
    \times \frac{1}{1 + \frac{\xi}{4\lambda}} \binom{2N}{N - r} 
\sum_{k=0}^{N-1} \frac{N!}{k! (N - 1 - k)!} (-1)^k \left( 1 - \frac{k}{\frac{\xi}{4\lambda} + 1 + k} \right) 
{}_2F_1\left( -2 (1 + k), r-N, -2N, 2 \right)
\nonumber\\
&& \hspace*{-1cm}
+ \left( \frac{1}{4} \right)^N 
\frac{\Gamma\left( 1 + \frac{\xi}{4\lambda} \right) \Gamma\left( N + \frac{1}{4} \left( 2 + \frac{\xi}{\lambda} \right) \right)}{
    \Gamma\left( 1 + N + \frac{\xi}{4\lambda} \right) \Gamma\left( \frac{1}{4} \left( 2 + \frac{\xi}{\lambda} \right) \right) 
    + \Gamma\left( 1 + \frac{\xi}{4\lambda} \right) \Gamma\left( N+ \frac{1}{4} \left( 2 + \frac{\xi}{\lambda} \right) \right)
    }
\nonumber\\
&& \hspace*{-1cm}
\times \frac{2}{1 + \frac{\xi}{2\lambda}} \binom{2N}{N - r} 
\sum_{k=0}^{N-1} \frac{N!}{k! (N - 1 - k)!} (-1)^k \left( 1 - \frac{k}{\frac{\xi}{4\lambda} + \frac{1}{2} + k} \right) 
{}_2F_1\left( -1 - 2k, r-N, -2N, 2 \right).
\nonumber
\end{eqnarray}
Let us consider the first sum in (\ref{rho_l_equal_mu_2}). Denoting by ${\cal P}_n^{(a,b)}(x)$ the Jacobi polynomial, 
due to Eq. (7.3.1.142) of \cite{Prudnikov3} and recalling Eq. (45.2.1) of \cite{Hansen}, we obtain
\begin{eqnarray*}
&& \hspace*{-1cm}
\sum_{k=0}^{N}{N \choose k} (-1)^k \frac{\xi}{2 k \lambda + \xi} \; {}_2F_1\left( -k, r-N, -N, 2 \right)
=\sum_{k=0}^{N}{N \choose k} (-1)^k \frac{\xi}{2 k \lambda + \xi}\,\frac{2^k \, k! \, {\cal P}_k^{( -k - r + N, -k + r)}(0)}{(-N)_k}
\nonumber\\
&& \hspace*{-1cm}
=\frac{\frac{\xi}{2L} \, \Gamma(1 + r) \, \Gamma\left(\frac{\xi}{2\lambda}\right)}{\Gamma\left(1 + r + \frac{\xi}{2\lambda}\right)} \, {}_2F_1\left(r- N, \frac{\xi}{2\lambda}, 1 + r + \frac{\xi}{2\lambda}, -1\right).
\label{sum1_fin}
\end{eqnarray*}
For the second sum in (\ref{rho_l_equal_mu_2}), by making use of Eqs. (5.14.2.12) and (2.21.1.4) of \cite{Prudnikov3}, we have
\begin{eqnarray*}
&& \hspace*{-1cm}
\sum_{k=0}^{N} \binom{N}{k} (-1)^k \, \frac{\xi}{4 k \lambda + \xi} \, {}_2F_1\left(-2 k, -N + r, -2 N, 2\right)
\nonumber\\
&& \hspace*{-1cm}
=\frac{\xi}{4\lambda}\int_0^1 y^{\frac{\xi}{4\lambda} - 1} \sum_{k=0}^{N} \frac{(2k)!}{\left(\frac{1}{2} - n\right)_k k!}   \left(\frac{y}{4}\right)^k \, {\cal P}_{2k}^{(-1 - 2N, -r + N - 2k)}( -3)\,\, {\rm d}y
\nonumber\\
&& \hspace*{-1cm}
=\frac{\xi}{4\lambda} \int_0^1 y^{\frac{\xi}{4\lambda} - 1} (1 + \sqrt{y})^{2N - r} (1 - \sqrt{y})^r \, {}_2F_1\left(r - N, -N, -2N, \frac{4\sqrt{y}}{(1 + \sqrt{y})^2}\right) \, {\rm d}y
\nonumber\\
&& \hspace*{-1cm}
=\frac{\xi}{4\lambda} \, \text{Beta}\left( \frac{\xi}{4\lambda}, r + 1 \right) \, {}_3F_2\left( \left\{ r - N, \frac{1}{2} + r, \frac{\xi}{4\lambda} \right\}, \left\{ \frac{1}{2} - N, \frac{\xi}{4\lambda} + r + 1 \right\}, 1 \right).
\label{sum2_fin}
\end{eqnarray*}
Similarly for the third sum in (\ref{rho_l_equal_mu_2}), we have:
\begin{eqnarray*}
&&\sum_{k=1}^{N} \frac{N!}{(k-1)!(N-k)!} (-1)^k \frac{1}{k + \frac{\xi}{4\lambda}} \, {}_2F_1\left(-2k, r-N, -2N, 2\right)
\nonumber\\
&&=-\frac{\xi}{4\lambda} \, \text{Beta}\left( \frac{\xi}{4\lambda}, r + 1 \right) \, {}_3F_2\left( \left\{ r - N, \frac{1}{2} + r, \frac{\xi}{4\lambda} \right\}, \left\{ \frac{1}{2} - N, \frac{\xi}{4\lambda} + r + 1 \right\}, 1 \right).
\label{sum3_fin}
\end{eqnarray*}
Finally, we consider the fourth sum in (\ref{rho_l_equal_mu_2}), for which, by making use of Eqs. (5.14.2.12) and (2.21.1.4) of \cite{Prudnikov3}, Eq. (109) of \cite{Brychkov2013} and Eq. (106) of \cite{Srivastava}, we obtain
\begin{eqnarray*}
&&\hspace{-0.8cm}
\sum_{k=1}^{N} \frac{N!}{(k-1)!(N-k)!} (-1)^k \frac{1}{k + \frac{\xi}{4\lambda} - \frac{1}{2}} \, {}_2F_1\left(1 - 2k, r-N, -2N, 2\right)
\\
&&\hspace{-0.8cm}
=-\sum_{k=1}^{N} \frac{2^{-2 + 2k}\left(\frac{1}{2}\right)_k}{\left(\frac{1}{2} - N\right)_k}  \frac{ 2N - 2k + 1}{k + \frac{\xi}{4\lambda} - \frac{1}{2}} \, {\cal P}_{2k-1}^{(1 - 2k - r+ N, 1 - 2k + r + N)} (0)
\nonumber\\
&&\hspace{-0.8cm}
=-r \int_0^1 y^{\frac{\xi}{4\lambda} - \frac{1}{2}}  
 (1 - \sqrt{y})^{-1 -r}(1 + \sqrt{y})^{-1 +r+2N} 
  {}_2F_1\left( -r - N, -N, -2 N, \frac{4 \sqrt{y}}{\left( 1 + \sqrt{y} \right)^2} \right)\,\, {\rm d}y
\nonumber\\
&&\hspace{-0.8cm}
=-r \, \text{Beta}\left( \frac{\xi}{4 \lambda} + \frac{1}{2}, r \right) \, {}_3F_2\left( \left\{ \frac{1}{2} + r, -N + r, \frac{\xi}{4 L} + \frac{1}{2} \right\}, \left\{ \frac{1}{2} - N, \frac{\xi}{4 \lambda} + r + \frac{1}{2} \right\}, 1 \right).
\end{eqnarray*}
Hence, the proof follows from straightforward calculations. 
\section{Approximation of $g(\lambda,\mu,\xi,N)$}
\label{proof_lemma_appr_g}
To achieve the result in Proposition \ref{Prop_lim_mean_var}, we first require an assessment of the behavior of $g(\lambda,\mu,\xi,N)$ for large $N$, which necessitates the approximation detailed in the following lemma.

\begin{lemma}
	If $\lambda<\mu$, the function $g(\lambda,\mu,N)$ defined in Eq. (\ref{g_1}) can be approximated for large $N$ as
	\begin{equation}
		\label{g_approx}
		\tilde g(\lambda,\mu,\xi,N)\approx \frac{c_5(\lambda,\mu,\xi)N^5+c_4(\lambda,\mu,\xi)N^4+c_3(\lambda,\mu,\xi)N^3+c_2(\lambda,\mu,\xi)N^2+c_1(\lambda,\mu,\xi)N}{d_5(\lambda,\mu,\xi)N^5+d_4(\lambda,\mu,\xi)N^4+d_3(\lambda,\mu,\xi)N^3+d_2(\lambda,\mu,\xi)N^2+d_1(\lambda,\mu,\xi)N},
	\end{equation}
	where the coefficients $c_i(\lambda,\mu,\xi)$ and $d_i(\lambda,\mu,\xi)$, $i=1,\ldots,5$, are given in Eqs. (\ref{coff_num_g}) and (\ref{coff_num_den_g}), respectively.
	\label{lemma_appross_g}
\end{lemma}
\begin{proof}
Recalling Theorem 1.1 of Daalhuis \cite{Daalhuis}, if $\lambda<\mu$ and $N$ is large, we have
	\begin{eqnarray}
		&&{}_{2}F_{1}\left(\frac{\xi}{\lambda+\mu}+i,-N,\frac{\xi}{\lambda+\mu}+N+1;-\frac{\lambda}{\mu}\right) \approx \frac{2^N \left(1+\frac{\lambda}{\mu}\right)^{N-\left(\frac{\xi}{\lambda+\mu}+i\right)}\Gamma\left(\frac{\xi}{\lambda+\mu}+N+1\right)\Gamma(N+1)}{\left(\frac{\lambda}{\mu}\right)^{N/2}\Gamma\left(\frac{\xi}{\lambda+\mu}+2N+1\right)\sqrt{2\pi}}\nonumber\\
		&&\times \left[N^{\frac{1}{2}\left(\frac{\xi}{2(\lambda+\mu)}-1+i\right)}D_{-\frac{\xi}{\lambda+\mu}-i}\left(-\alpha(\lambda, \mu)\sqrt{N}\right)\left(\gamma_{00}^{i}+\frac{\gamma_{01}^{i}}{n}\right)
		\right.\nonumber\\
		&&\left.+
		N^{\frac{1}{2}\left(\frac{\xi}{2(\lambda+\mu)}-2+i\right)}D_{-\frac{\xi}{\lambda+\mu}+1-i}\left(-\alpha(\lambda, \mu)\sqrt{N}\right)\left(\gamma_{10}^{i}+\frac{\gamma_{11}^{i}}{n}\right)
		\right],\quad i=0,1,\nonumber
	\end{eqnarray}
	where $D_r(z)$ is the parabolic cylinder function, $\alpha(\lambda, \mu):=- \sqrt{2 \log\left( \frac{\left(\lambda+\mu\right)^2}{4 \lambda \mu} \right)}$, and the coefficients 
	$\gamma_{kj}^{i}$ ($i,j,k=0,1$) are given by
	\begin{eqnarray}
		&&\gamma_{00}^{0}(\lambda, \mu, \xi)=\left(\frac{ \lambda}{\mu}+1\right)\left(\frac{\lambda^2-\mu^2}{\lambda \,\mu\, \alpha(\lambda,\mu)}\right)^{\frac{\xi}{\lambda+\mu}}\left(\frac{\alpha(\lambda,\mu)\mu}{\lambda-\mu}\right),\nonumber\\
		&&\gamma_{01}^{0}(\lambda, \mu, \xi)=\frac{\lambda + \mu - \xi}{(\lambda^2 - \mu^2)\alpha(\lambda, \mu)} 
		\left[
		\left( 
		\frac{\lambda^2 - \mu^2}{\alpha(\lambda, \mu)\lambda \mu}
		\right)^{\frac{\xi}{\lambda + \mu}} 
		\frac{
			(\alpha(\lambda, \mu))^2	\left( 4\lambda\mu(\lambda + \mu) - (\lambda^2 + \mu^2)\xi \right)
			+ (\lambda - \mu)^2 \xi
		}{
			2(\lambda - \mu)^2
		}\right.\nonumber\\
		&&\hspace{1cm}\left.-
		2^{\frac{1}{2} + \frac{\xi}{\lambda + \mu}}  
		\frac{\lambda - \mu}{\alpha(\lambda, \mu)} 
		\left( 
		\frac{\alpha(\lambda, \mu)(\lambda + \mu)}{\lambda - \mu}
		\right)^{\frac{\xi}{\lambda + \mu}} 
		\right],
		\nonumber\\
		&&\gamma_{10}^{0}(\lambda, \mu, \xi)=\frac{(\lambda + \mu)^{\frac{\xi}{\lambda + \mu}}}{\alpha(\lambda, \mu)}
		\left[
		\left(\frac{(\lambda+\mu)\alpha(\lambda,\mu)}{\lambda-\mu}\right) \left( 
		\frac{\lambda - \mu}{\alpha(\lambda, \mu)\lambda \mu } 
		\right)^{\frac{\xi}{\lambda + \mu}}
		- 
		2^{\frac{\xi}{\lambda + \mu} + \frac{1}{2}} 
		\left( 
		\frac{\alpha(\lambda, \mu)}{\lambda - \mu} 
		\right)^{\frac{\xi}{\lambda + \mu} }
		\right],\nonumber\\
		&&\gamma_{11}^{0}(\lambda, \mu, \xi)=\frac{1}{2 \sqrt{2} \, \alpha(\lambda, \mu) (\lambda - \mu)^3 (\lambda + \mu)^2}\nonumber\\
		&&\hspace{1cm}\times\left[
		\left( \frac{2(\lambda + \mu)\alpha(\lambda, \mu)}{\lambda - \mu} \right)^{\frac{\xi}{\lambda + \mu}} \right.\nonumber\\
		&&\hspace{1cm}\times
		\left( 
		\frac{2(\lambda - \mu)^3(\lambda + \mu - \xi)(-2(\lambda + \mu) + \xi)}{\alpha(\lambda, \mu)^2}
		+ 
		\frac{(\lambda - \mu)(\lambda + \mu)^2((\lambda - \mu)^2 + 2(\lambda + \mu)\xi + 2\xi^2)}{2}
		\right)\nonumber\\
		&&\hspace{1cm}\left.+ \sqrt{2} (\lambda + \mu) 
		\left( \frac{\lambda^2 - \mu^2}{\alpha(\lambda, \mu) \lambda \mu} \right)^{\frac{\xi}{\lambda + \mu}} 
		\left( 
		\alpha(\lambda, \mu)(\lambda + \mu - \xi)(4\lambda \mu (\lambda + \mu) - (\lambda^2 + \mu^2)\xi)
		- 
		\frac{(\lambda - \mu)^2 \xi (\lambda + \mu + \xi)}{\alpha(\lambda, \mu)}
		\right)
		\right],
		\nonumber
		\end{eqnarray}
		and 
		\begin{eqnarray}
		&&\gamma_{00}^{1}(\lambda, \mu, \xi)=\left(\frac{ \lambda}{\mu}+1\right)\left(\frac{\lambda^2-\mu^2}{\lambda \,\mu\, \alpha(\lambda,\mu)}\right)^{\frac{\xi}{\lambda+\mu}}\nonumber\\
		&&\gamma_{01}^{1}(\lambda, \mu, \xi)=\frac{\xi}{(\lambda - \mu)\alpha(\lambda, \mu)}\nonumber\\
		&&\hspace{1cm}\times \left[
		\frac{
			(\alpha(\lambda, \mu))^2 \left( 
			\lambda^3 - 3\lambda\mu^2 + (\lambda^2+\mu^2)( \xi-\mu ) 
			\right)
			- (\lambda - \mu)^2 (\lambda + \mu + \xi)
		}{
			2(\lambda^2 - \mu^2)\mu \alpha(\lambda, \mu)
		}
		\left( 
		\frac{\lambda^2 - \mu^2}{\alpha(\lambda, \mu)\lambda\mu}
		\right)^{\frac{\xi}{\lambda + \mu}} 
		\right.\nonumber\\
		&&\hspace{1cm}\left.+ 
		2^{\frac{1}{2} + \frac{\xi}{\lambda + \mu}} 
		\left( 
		\frac{\alpha(\lambda, \mu)(\lambda + \mu)}{\lambda - \mu}
		\right)^{\frac{\xi}{\lambda + \mu}}\right],
		\nonumber\\
		&&\gamma_{10}^{1}(\lambda, \mu, \xi)=\frac{(\lambda + \mu)^{\frac{\xi}{\lambda + \mu} + 1}}{\alpha(\lambda, \mu)}
		\left[
		\frac{1}{\mu} \left( 
		\frac{\lambda - \mu}{\alpha[\lambda, \mu]\lambda \mu } 
		\right)^{\frac{\xi}{\lambda + \mu}}
		- 
		2^{\frac{\xi}{\lambda + \mu} + \frac{1}{2}} 
		\left( 
		\frac{\alpha(\lambda, \mu)}{\lambda - \mu} 
		\right)^{\frac{\xi}{\lambda + \mu} + 1}
		\right],\nonumber\\
		&&\gamma_{11}^{1}(\lambda, \mu, \xi)=\frac{1}{2 \alpha(\lambda, \mu)(\lambda + \mu)} 
		\left[
		\left( \frac{\alpha(\lambda, \mu)(\lambda + \mu)}{\lambda - \mu} \right)^{\frac{\xi}{\lambda + \mu}} 
		\frac{1}{\lambda - \mu}\right.\nonumber\\
		&&\hspace{1cm}\left.\times
		\left(
		\frac{
			2^{\frac{1}{2} + \frac{\xi}{\lambda + \mu}} (\lambda + \mu - \xi) \xi
		}{\alpha(\lambda, \mu)
		}
		+ 
		\frac{2^{-\frac{3}{2} + \frac{\xi}{\lambda + \mu}} (\lambda + \mu)^2 \alpha(\lambda, \mu) \left( (3\lambda + \mu)^2 + 2(5\lambda + \mu)\xi + 2\xi^2 \right)
		}{(\lambda - \mu)^2}\right)\right.\nonumber\\
		&&\hspace{1cm}\left.-
		\left( \frac{\lambda^2 - \mu^2}{\alpha(\lambda, \mu)\lambda \mu} \right)^{\frac{\xi}{\lambda + \mu}} 
		\frac{1}{\mu}
		\left(\frac{
			(\lambda + \mu + \xi)(2(\lambda + \mu) + \xi)
		}{(\alpha(\lambda, \mu))^2}
		+
		\frac{
			\xi\left( -\lambda^3 + 3\lambda\mu^2 + (\lambda^2 + \mu^2)(\mu - \xi) \right)
		}{
			(\lambda - \mu)^2
		}
		\right)
		\right].
		\nonumber
	\end{eqnarray}
	Hence, recalling that $D_r(z) \approx z^r e^{-\frac{z^2}{4}}\left[1-\frac{(r-1)r}{2z^2}+\frac{(r-3)(r-2)(r-1)r}{8z^4}- \frac{(r - 5)(r - 4)(r - 3)(r - 2)(r - 1)\, r}{32\, z^6}
	\right]$ for $z\rightarrow \infty$, we obtain (\ref{g_approx}) by setting
\begin{eqnarray}
	&&c_5(\lambda,\mu,\xi)=-32 \sqrt{2}\, (\lambda + \mu)^7\, \alpha(\lambda, \mu)^7 \left[
	\gamma_{00}^0(\lambda, \mu, \xi) - \alpha(\lambda, \mu) \, \gamma_{10}^0(\lambda, \mu, \xi)
	\right],
	\nonumber\\
	&&c_4(\lambda,\mu,\xi)=-16 \sqrt{2} \, (\lambda + \mu)^5 \, \alpha(\lambda, \mu)^5 \,
	\Big[
	- \xi (\lambda + \mu + \xi) \, \gamma_{00}^0(\lambda, \mu, \xi)
	+ 2 (\lambda + \mu)^2 \, \alpha(\lambda, \mu)^2 \, \gamma_{01}^0(\lambda, \mu, \xi)\nonumber \\
	&&\hspace{1cm}- \alpha(\lambda, \mu) \left(
	(\lambda + \mu - \xi)\, \xi\, \gamma_{10}^0(\lambda, \mu, \xi)
	+ 2 (\lambda + \mu)^2 \, \alpha(\lambda, \mu)^2 \, \gamma_{11}^0(\lambda, \mu, \xi)
	\right)
	\Big],
	\nonumber\\
	&&c_3(\lambda,\mu,\xi)=-4 \sqrt{2} \, (\lambda + \mu)^3 \, \xi \, \alpha(\lambda, \mu)^3 \,
	\Big[
	(\lambda + \mu + \xi)\, (2(\lambda + \mu) + \xi)\, (3(\lambda + \mu) + \xi)\, \gamma_{00}^0(\lambda, \mu, \xi),\nonumber \\
	&&\hspace{1cm}- 4 (\lambda + \mu)^2 (\lambda + \mu + \xi)\, \alpha(\lambda, \mu)^2\, \gamma_{01}^0(\lambda, \mu, \xi)\nonumber \\
	&&\hspace{1cm}+ \alpha(\lambda, \mu)\, (\lambda + \mu - \xi)
	\left(
	(\lambda + \mu + \xi)(2(\lambda + \mu) + \xi)\, \gamma_{10}^0(\lambda, \mu, \xi)
	- 4 (\lambda + \mu)^2\, \alpha(\lambda, \mu)^2\, \gamma_{11}^0(\lambda, \mu, \xi)
	\right)
	\Big],
	\nonumber\\
	&&c_2(\lambda,\mu,\xi)= \sqrt{2}\, (\lambda + \mu) \, \xi \, (\lambda + \mu + \xi) \, (2(\lambda + \mu) + \xi) \, \alpha(\lambda, \mu) \nonumber\\
	&&\hspace{1cm}\times	\Big[
	(3(\lambda + \mu) + \xi)(4(\lambda + \mu) + \xi)(5(\lambda + \mu) + \xi) \gamma_{00}^0(\lambda, \mu, \xi)\nonumber \\
	&&\hspace{1cm}	- 4 (\lambda + \mu)^2 (3(\lambda + \mu) + \xi) \alpha(\lambda, \mu)^2 \gamma_{01}^0(\lambda, \mu, \xi)\nonumber \\
	&&\hspace{1cm}	+ (\lambda + \mu - \xi) \alpha(\lambda, \mu) \big(
	(3(\lambda + \mu) + \xi)(4(\lambda + \mu) + \xi) \gamma_{10}^0(\lambda, \mu, \xi) \nonumber\\
	&&\hspace{1cm}	- 4 (\lambda + \mu)^2 \alpha(\lambda, \mu)^2 \gamma_{11}^0(\lambda, \mu, \xi)
	\big)
	\Big],\nonumber\\
	&&c_1(\lambda,\mu,\xi)=\sqrt{2} (\lambda + \mu)\, \xi\, (\lambda + \mu + \xi)\, \left(2(\lambda + \mu) + \xi\right)\, \left(3(\lambda + \mu) + \xi\right)\, \left(4(\lambda + \mu) + \xi\right)\,
	\alpha(\lambda,\mu),  \nonumber\\
	&&\hspace{1cm}\times
	\left[(5(\lambda + \mu) + \xi)\, \gamma_{01}^0(\lambda, \mu, \xi)
	+\alpha(\lambda,\mu) \, (\lambda + \mu - \xi)\,  \gamma_{11}^0(\lambda, \mu, \xi) \right],
	\label{coff_num_g}
	\end{eqnarray}
	and
	\begin{eqnarray}
	&&d_5(\lambda,\mu,\xi)=32\sqrt{2}\, \mu\, (\lambda + \mu)^6\, \alpha(\lambda, \mu)^6 \left[
	\, \gamma_{00}^1(\lambda, \mu, \xi)
	- \alpha(\lambda, \mu)\, \gamma_{10}^1(\lambda, \mu, \xi)
	\right],
	\nonumber\\
	&&d_4(\lambda,\mu,\xi)=-16 \sqrt{2}\, \mu\, (\lambda + \mu)^4\, \alpha(\lambda, \mu)^4 \,\nonumber\\
	&&\hspace{1cm}\times \Big[  (\lambda + \mu + \xi)\, (2(\lambda + \mu) + \xi)\, \gamma_{00}^1(\lambda, \mu, \xi) \nonumber\\
	&&\hspace{1cm} - 2(\lambda + \mu)^2\, \alpha(\lambda, \mu)^2\, \gamma_{01}^1(\lambda, \mu, \xi) \nonumber\\
	&&\hspace{1cm}- \alpha(\lambda, \mu) \left(
	\xi(\lambda + \mu + \xi)\, \gamma_{10}^1(\lambda, \mu, \xi)
	- 2(\lambda + \mu)^2\, \alpha(\lambda, \mu)^2\, \gamma_{11}^1(\lambda, \mu, \xi)
	\right)
	\Big],\nonumber\\
	&&d_3(\lambda,\mu,\xi)=4 \sqrt{2} \, \mu \, (\lambda + \mu)^2 \, (\lambda + \mu + \xi) \, \alpha(\lambda, \mu)^2 \,\nonumber\\
	&&\hspace{1cm}\times \Big[  (2(\lambda + \mu) + \xi)(3(\lambda + \mu) + \xi)(4(\lambda + \mu) + \xi)\, \gamma_{00}^1(\lambda, \mu, \xi), \nonumber\\
	&&\hspace{1cm} - 4 (\lambda + \mu)^2 (2(\lambda + \mu) + \xi)\, \alpha(\lambda, \mu)^2\, \gamma_{01}^1(\lambda, \mu, \xi) \nonumber\\
	&&\hspace{1cm} - \xi\, \alpha(\lambda, \mu) \left(
	(2(\lambda + \mu) + \xi)(3(\lambda + \mu) + \xi)\, \gamma_{10}^1(\lambda, \mu, \xi)
	- 4 (\lambda + \mu)^2\, \alpha(\lambda, \mu)^2\, \gamma_{11}^1(\lambda, \mu, \xi)
	\right)
	\Big],\nonumber\\
	&&d_2(\lambda,\mu,\xi)=- \sqrt{2}\, \mu(\lambda + \mu + \xi)\, (2(\lambda + \mu) + \xi)\, (3(\lambda + \mu) + \xi) \nonumber\\
	&&\hspace{1cm}\times\Big[ 
	(4(\lambda + \mu) + \xi)(5(\lambda + \mu) + \xi)(6(\lambda + \mu) + \xi)\, \gamma_{00}^1(\lambda, \mu, \xi)\nonumber\\
	&&\hspace{1cm}- 4 (\lambda + \mu)^2 (4(\lambda + \mu) + \xi)\, \alpha(\lambda, \mu)^2\, \gamma_{01}^1(\lambda, \mu, \xi)\nonumber\\
	&&\hspace{1cm}\quad - \xi\, \alpha(\lambda, \mu) \left(
	(4(\lambda + \mu) + \xi)(5(\lambda + \mu) + \xi)\, \gamma_{10}^1(\lambda, \mu, \xi)
	- 4 (\lambda + \mu)^2\, \alpha(\lambda, \mu)^2\, \gamma_{11}^1(\lambda, \mu, \xi)
	\right)
	\Big],\nonumber\\
	&&d_1(\lambda,\mu,\xi)=	- \sqrt{2}\, \mu\, 
	(\lambda + \mu + \xi)\, 
	(2(\lambda + \mu) + \xi)\, 
	(3(\lambda + \mu) + \xi)\, 
	(4(\lambda + \mu) + \xi)\, 
	(5(\lambda + \mu) + \xi) \nonumber\\
	&&\hspace{1cm}	\times	\left[
	(6(\lambda + \mu) + \xi)\, \gamma_{01}^1(\lambda, \mu, \xi)
	- \xi\, \alpha(\lambda, \mu)\, \gamma_{11}^1(\lambda, \mu, \xi)
	\right].
	\label{coff_num_den_g}
\end{eqnarray}
\end{proof}

\begin{table}[t]
	\centering
\begin{tabular}{c||l|l||l|l||l|l}
$N$ &  \multicolumn{2}{|c||}{$\xi=0$}   &\multicolumn{2}{|c||}{$\xi=0.2$} & \multicolumn{2}{|c}{$\xi=0.5$} \\
\hline
      & $g$ & ${\tilde g}$ &$g$ & ${\tilde g}$ &$g$ & ${\tilde g}$ \\
\hline
$100$	& $0.420$ &$0.422$ &$0.422$&$0.424$&$0.423$&$0.426$ \\
			\hline
$1000$	&$0.402359$&$0.402360$ &$0.402505$ &$0.402506$&$0.402724$&$0.402725$ \\
\hline
$5000$	&$0.40047833$& $0.40047834$ &$0.40050816$& $0.40050817$ &$0.40055290$&$0.40055291$\\
			\hline
		\end{tabular}
\caption{$\lambda=0.1,\,\mu=2$, $\xi=0,0.2,0.5$, for varying $N$. }
\label{tabellaapprossg}
\end{table}

In order to evaluate the goodness of the numerical approximation obtained in Lemma \ref{lemma_appross_g} for $g(\lambda,\mu,N)$, 
we compare in Table \ref{tabellaapprossg}  its exact values defined in Eq. (\ref{g_1}), with the corresponding approximated quantities
$\tilde g(\lambda,\mu,N)$ as $N$ increases and for different values of $\xi$. 
The given values confirm that the approximation improves as a $N\to \infty$ and that the approximation is satisfactory when $N$ is large.

\section{Asymptotic variance}\label{app_lim_var}
The behaviour of asymptotic variance when $\lambda<\mu$ and $N \rightarrow +\infty$  is attained from Proposition \ref{mean_var_asym},  by applying the approximation in Eq. (\ref{g_approx}); hence it results:
\begin{eqnarray}
&&\lim_{N \rightarrow +\infty} Var[{\cal N}]=\frac{1}{(d_5(\lambda, \mu, \xi))^2\, (\lambda + \mu + \xi)^2\, (2\lambda + 2\mu + \xi)}
\nonumber\\
&&\hspace{2.5cm}\times \left[
\left( (d_4(\lambda, \mu, \xi))^2 + 2\, d_3(\lambda, \mu, \xi)\, d_5(\lambda, \mu, \xi) \right) \xi (\lambda - \mu)^2 \right.
\nonumber\\
&&\hspace{2.5cm}
- \left( 2\, c_3(\lambda, \mu, \xi)\, c_5(\lambda, \mu, \xi) + (c_4(\lambda, \mu, \xi))^2 \right) \mu^2 (2\lambda + 2\mu + \xi) 
\nonumber\\
&&\hspace{2.5cm}
- \left( c_5(\lambda, \mu, \xi)\, d_4(\lambda, \mu, \xi) + c_4(\lambda, \mu, \xi)\, d_5(\lambda, \mu, \xi) \right) \mu (2\lambda + \xi)(\lambda + \mu + \xi) 
\nonumber\\
&&\hspace{2.5cm}
-  \left( c_4(\lambda, \mu, \xi)\, d_4(\lambda, \mu, \xi) +
c_3(\lambda, \mu, \xi)\, d_5(\lambda, \mu, \xi) +
c_5(\lambda, \mu, \xi)\, d_3(\lambda, \mu, \xi)
\right)2\mu(\lambda^2 - \mu^2)
\nonumber\\
&&\hspace{2.5cm}\left.+ 2\, d_4(\lambda, \mu, \xi)\, d_5(\lambda, \mu, \xi) 
\left((4 \lambda \mu + \xi^2)(\lambda + \mu) + \lambda^2 \xi + \mu^2 \xi + 6 \lambda \mu \xi
\right)\right],
\nonumber
\end{eqnarray}
where the coefficients $c_i(\lambda,\mu,\xi)$ and $d_i(\lambda,\mu,\xi)$, $i=1,\ldots,5$, are given in Eqs. (\ref{coff_num_g}) and (\ref{coff_num_den_g}), respectively.

\end{document}